\documentclass[12pt]{amsart}
\usepackage{amsmath, amssymb}

\newcommand{\cor}[1]{{\! \!}}

\numberwithin{equation}{section}

\theoremstyle{plain}
\newtheorem*{t-theorem}{Theorem}
\newtheorem{theorem}{Theorem}[section]

\newtheorem{corollary}[theorem]{Corollary}
\newtheorem{lemma}[theorem]{Lemma}

\theoremstyle{definition}
\newtheorem{remark}[theorem]{Remark}
\newtheorem{definition}[theorem]{Definition}

\def \N {\mathbb{N}}
\def \R {\mathbb{R}}

\def \E {\mathbb{E}}
\def \P {\mathbb{P} \, }

\def \WW {\mathcal{W}}
\def \LL {\mathcal{L}}
\def \NN {\mathcal{N}}

\def \a {\alpha}
\def \b {\beta}

\def \G {\Gamma}
\def \e {\varepsilon}
\def \d {\delta}
\def \D {\Delta}

\def \k {\kappa}

\def \s {\sigma}
\def \sf {\text{\tt s}}
\def \t {\tau}

\def \Om {\Omega}

\def \etc {,\ldots,}

\def \dist {{\rm dist}}

\def \dim {{\rm dim}}
\def \supp {{\rm supp}}

\newcommand{\norm}[1]{\left \| #1 \right \|}

\title[Singular values and permanent estimators]{Singular values of Gaussian matrices and permanent estimators}

\author{Mark Rudelson}\thanks{M.R.: Department of Mathematics, University of Michigan. Partially supported by NSF grant DMS 1161372 and USAF Grant FA9550-14-1-0009.}
\author{Ofer Zeitouni}
\thanks{O. Z.: Faculty of Mathematics, Weizmann Institute and
Courant Institute, New York University. Partially supported by NSF grant DMS 1106627 and a grant from the Israel Science Foundation.}

\date{January 26, 2013. Revised December 27, 2013}

\begin{document}
\maketitle
\begin{abstract}
 We present estimates on the small singular values of a class of matrices
 with independent Gaussian entries and inhomogeneous variance profile,
 satisfying a broad-connectedness condition. Using these estimates
 and concentration of measure for the spectrum of Gaussian matrices with
 independent entries, we prove that for a large class of
 graphs satisfying an appropriate expansion property, the
 Barvinok--Godsil-Gutman estimator for the permanent achieves
 sub-exponential errors with high probability.
\end{abstract}
\section{Introduction}
Recall that the permanent of an $n$-by-$n$  matrix $A$ is defined as
$$\mbox{\rm per}(A)=\sum_{\pi\in \mathcal{S}_n} a_{1,\pi(1)}a_{2,\pi(2)}
\cdots a_{n,\pi(n)}\,,$$
where the summation is over all permutations of $n$ elements.
In this paper we consider only matrices $A$ with non-negative entries. This
includes in particular matrices with 0--1 entries, for which the evaluation
of the permanent is fundamental in combinatorial counting problems.
For general
0--1 matrices, the evaluation of the permanent is a $\#P$-complete
problem \cite{Val}. Thus, the interest is in obtaining algorithms
that compute approximations to the permanent, and indeed
a polynomial running time
Markov Chain Monte Carlo
randomized algorithm that evaluates
$\mbox{\rm per}(A)$ (up to $(1+\epsilon)$ multiplicative errors, with
complexity polynomial in $\epsilon$) is available
\cite{JSV}. In practice, however, the running time of such an algorithm, which
is $O(n^{10})$,
still makes it challenging to implement for large $n$. (An alternative, faster
MCMC algorithm is presented in \cite{MRchange}, with claimed running time
of $O(n^7(\log n)^4)$.)

An earlier
simple probabilistic algorithm for the evaluation of $\mbox{\rm per}(A)$
is based on the following observation: if $x_{i,j}$ are i.i.d. zero mean
variables with unit variance
and $X$ is an $n\times n$ matrix with entries $x_{i,j}$, then
an easy computation shows that
\begin{equation}
  \label{GG}
  \mbox{\rm per}(A)=\E(\mbox{\rm det}(A_{1/2}\odot X))^2\,,
\end{equation}
where for any two $n\times m$ matrices
$A,B$,
$D=A \odot B$ denotes their Hadamard, or Schur, product,
i.e. the $n \times m$ matrix with entries
$d_{i,j}=a_{i,j} \cdot b_{i,j}$, and where $A_{1/2}(i,j)=A(i,j)^{1/2}$.
Thus, $\mbox{\rm det}(A_{1/2}\odot X)^2$ is an unbiased estimator of $\text{per}(A)$.
This algorithm was proposed (with $x_{i,j}\in \{-1,1\}$) in \cite{GG}, and
takes advantage of the fact that the
evaluation of determinants is computationally easy via Gaussian
elimination. While we do not discuss computational issues
in this article, we note that the evaluation of the determinant requires
at most $o(n^3)$ arithmetic operations; in terms of bit complexity,
for matrices with integer entries of $k$ bits, there exist algorithms
with complexity $O(n^\alpha k^{1+o(1)})$, with $\alpha<3$, see e.g.
\cite{KV} for a review and the value $\alpha\sim 2.7$.
To avoid rounding errors in the case of real valued random variables
one needs to take $k=n^{1+o(1)}$, yielding a total bit-complexity in that case
smaller than
$o(n^4)$.

Thus, the main question concerning the above algorithm
is the approximation error, and in particular
the concentration of the random variable
$\mbox{\rm det}^2(A_{1/2}\odot X)$ around its mean. For general matrices
$A$ with non-negative entries, Barvinok showed that using standard
Gaussian variables $x_{i,j}$, with probability approaching one, the resulting
multiplicative error is at most exponential in $n$, with sharp constant.
(The constant cannot be improved, as the example of $A$ being
the identity matrix
shows.)

For restricted classes of matrices, better performance is possible.
Thus, in \cite{FJ}, the authors analyzed a variant of the
Godsil-Gutman algorithm due to \cite{KKLLL} and showed that for certain
dense, {\it random} $0-1$ matrices, a multiplicative $(1+\epsilon)$ error
is achieved in time $O(n\omega(n)\epsilon^{-2})$.
In \cite{FRZ}, it is shown that for a restricted class of
non-random matrices, the
performance  achieved by the Barvinok--Godsil-Gutman
estimator is better than in the worst-case scenario.
(Here and
in the rest of this paper, we will refer to the above permanent estimator
with Gaussian entries as the Barvinok--Godsil-Gutman estimator.)
Indeed, if for some fixed constants $\alpha,\beta>0$
one has $a_{i,j}\in [\alpha,\beta]$, then
for any $\delta>0$, with $G$ denoting the standard Gaussian matrix,
$$ \P\left(\frac1n \left| \log\frac{\mbox{\rm det}(A_{1/2}\odot
G)^2}{\mbox{\rm per}(A)}
\right|>\delta\right)\to_{n\to\infty} 0\,,$$
uniformly in $A$; that is, for such matrices
this  estimator
achieves subexponentional (in $n$) errors, with $o(n^3)$ (arithmetic)
running time.
An improved analysis in presented in \cite{CV}, where it is shown  that the
approximation error in the same set of matrices
 is only exponential in $n^{2/3} \log n$.

The class of matrices considered in \cite{FRZ} is somewhat restricted - first,
it does not include incidence matrices of non-trivial graphs, and second,
for such matrices,
as noted in \cite{FRZ},  a polynomial error
deterministic algorithm with running time $O(n^4)$ is available by adapting the
algorithm in \cite{LSW}.
Our goal in this paper is to better understand the properties
of the Barvinok--Godsil-Gutman estimator, and
show that in fact the same analysis applies
for a class of matrices that arise from $(\delta,\kappa)$-{\em broadly
connected}
graphs, i.e. graphs with good expansion properties (see Definition
\ref{def-scon} for a precise definition).
Our first main result concerning permanent estimators  reads as follows.
\begin{theorem} \label{th: graph permanent-intro}
 There exist $C, C', c$ depending only on $\d, \k$ such that for any $\t \ge 1$
and any adjacency matrix $A$ of a $(\d,\k)$-broadly connected graph,
\begin{align}
  \label{eq-120912}
  & \P\left( \left|
   \log{\mbox{\rm det}^2(A_{1/2}\odot G)}
   -\E\log{\mbox{\rm det}^2(A_{1/2}\odot G)}
 \right| >C(\t n \log n)^{1/3} \right) \notag \\
 &\leq \exp(-\t)+ \exp \left(-c \sqrt{n} /\log n  \right).
\end{align}
and
\[
  \E\log{\mbox{\rm det}^2(A_{1/2}\odot G)}
  \le \log \text{\rm per}(A)
   \le \E\log{\mbox{\rm det}^2(A_{1/2}\odot G)} + C' \sqrt{n \log n}.
\]
\end{theorem}
For a more refined probability bound see Theorem \ref{th: graph permanent}. Combining the two inequalities of Theorem \ref{th: graph permanent-intro}, we obtain the concentration of the Barvinok-Godsil-Gutman estimator around the permanent.
\begin{corollary}
 Under the assumptions of Theorem \ref{th: graph permanent-intro},
\[
  \P\left( \left| \log\frac{\mbox{\rm det}^2(A_{1/2}\odot G)}
{\mbox{\rm per}(A)} \right| >  2 C' \sqrt{n \log n}  \right)
\leq \exp \left(-c \sqrt{n}/ \log n  \right).
\]
\end{corollary}

This corollary implies the uniform convergence in probability if we consider a family of  $(\delta,\kappa)$-broadly connected
$n\times n$ bipartite graphs with $n \to \infty$.

\begin{corollary}
  \label{cor-1}
   Let $\text{SC}_{\delta,\kappa,n}$ denote the collection of
adjacency matrices of $(\delta,\kappa)$-broadly connected
$n\times n$ bipartite graphs.
  Let $\{\t_n\}_{n=1}^\infty$ be a sequence of positive numbers such that $\t_n \to \infty$.
Set $\sf_n=  \t_n \sqrt{n \log n}$.
  Then for any $\e>0$,
  \begin{equation}
  \label{eq-per1}
    \lim_{n\to\infty} \sup_{A\in \text{\rm SC}_{\delta,\kappa,n}}
    \P\left(\frac1{\sf_n} \left|
    \log\frac{\mbox{\rm det}^2(A_{1/2}\odot G)}{\mbox{\rm per}(A)}
\right|
    >\e \right)=0\,.
  \end{equation}
\end{corollary}
We remark that the error estimate \eqref{eq-per1}
in Corollary \ref{cor-1} is probably
not optimal. Indeed, in the special case $A_{i,j}\equiv 1$, a consequence
of the distributional results concerning matrices with i.i.d.
Gaussian entries \cite{Wilks}, see also
\cite{Goodman}, is that
\eqref{eq-per1} holds with $\sf_n$ satisfying $\sf_n/\log n\to\infty$.
As Theorem \ref{th: graph permanent-intro} shows, the main source of error
is the discrepancy
between $\E \log \mbox{\rm det}^2(A_{1/2}\odot G)$ and
$ \log \E \mbox{\rm det}^2(A_{1/2}\odot G)$.

Our second main result pertains to graphs whose
adjacency matrix $A$ satisfies $\text{per}(A)>0$. For such
matrices, there exists a (polynomial time) scaling algorithm that transforms
$A$ into an (almost) doubly stochastic matrix, see \cite[Pages 552-553]{LSW}.
In particular,
there exists a deterministic algorithm (with running time $O(n^4)$)
that outputs non-negative diagonal matrices $D_1,D_2$ so that
$B=D_1AD_2$ is an approximately doubly stochastic matrix, i.e.
$\sum_i B_{i,j}\in [1/2,2] , \sum_j B_{i,j}\in [1/2,2]$.
(Much more can be achieved, but we do not use that fact.)
Since $\text{per}(A)=\text{per}(B)\cdot \prod_i (D_1(i,i)D_2(i,i))$,
evaluating $\text{per}(A)$ thus reduces to the evaluation of
$\text{per}(B)$.
 The properties of the Barvinok--Godsil-Gutman estimator
 for
 such matrices are given in the following theorem.
\begin{theorem} \label{th: matrix permanent-intro}
 Let $r>0$.
 There exist
 $c, C, C'$ depending only on $r, \d, \k$ with the following property.
 Let $0\leq  b_n\leq n$ be a given sequence.
 Let $B$ be an $n \times n$ matrix with entries
 $0 \le b_{i,j}\leq b_n/n$
such that
 \begin{align*}
    \sum_{i=1}^n b_{i,j} \le 1 \qquad &\text{for all } j \in [n]; \\
    \sum_{j=1}^n b_{i,j} \le 1 \qquad &\text{for all } i \in [n].
 \end{align*}
 Define the bipartite graph $\Gamma=\Gamma_B$ connecting the
vertices $i$ and $j$ whenever $b_{i,j} \ge r/n$,
 and assume that $\Gamma$ is   $(\delta,\kappa)$-broadly connected.
 Then for any
 $\t \geq 1$
\begin{align}
  \label{eq-120912b}
  & \P\left( \left|
   \log{\mbox{\rm det}^2(B_{1/2}\odot G)}
   -\E\log{\mbox{\rm det}^2(B_{1/2}\odot G)}
 \right| >
 C (\t b_n n)^{1/3} \log^{c'} n\right) \notag \\
&\leq \exp(-\t)+ \exp(-C \log^4 n)
\end{align}
and
\begin{multline} \label{eq: perm-expectation-matrices}
  \E\log{\mbox{\rm det}^2(B_{1/2}\odot G)}
  \le \log \text{\rm per}(B)
   \le \E\log{\mbox{\rm det}^2(B_{1/2}\odot G)}
    + C' \sqrt{b_n n} \log^{c'} n.
\end{multline}
\end{theorem}
  As  in Theorem \ref{th: graph permanent-intro}, we can derive the concentration around the permanent and the uniform convergence in probability.
\begin{corollary} \label{cor: matrix perm conc}
 Under the assumptions of Theorem \ref{th: matrix permanent-intro},
\[
  \P\left( \left| \log\frac{\mbox{\rm det}^2(B_{1/2}\odot G)}
{\mbox{\rm per}(B)} \right| > 2  C' \sqrt{b_n n} \log^{c'} n \right)
   \leq \exp(-C \log^4 n).
\]
\end{corollary}

\begin{corollary}
  \label{cor-2}
  Let $\text{GSC}_{c,\delta,\kappa,n}$ denote the collection
  of $n \times n$ matrices $B$ with properties as
  in Theorem \ref{th: matrix permanent-intro}.
   Then there exists a constant
  $\bar C=\bar C(c,\delta,\kappa)$ so that with
  $\sf_n=(n b_n \log^{\bar C} n )^{1/2} $,
  and any $\e>0$,
  \begin{equation*}
    \lim_{n\to\infty} \sup_{B\in \text{\rm GSC}_{c,\delta,\kappa,n}}
    \P\left(\frac1{\sf_n} \left|
    \log\frac{\mbox{\rm det}^2(B_{1/2}\odot G)}{\mbox{\rm per}(B)}
\right|
    >\e \right)=0\,.
  \end{equation*}
\end{corollary}
Corollary \ref{cor: matrix perm conc} applies, in particular, to approximately
doubly stochastic matrices $B$
whose entries satisfy $c/n \le b_{i,j} \le 1$ for all $i,j$.
For such matrices the graph $\Gamma_A$ is complete, so the broad
connectedness condition is trivially satisfied.
Note that if such matrix contains entries of order
$\Omega(1)$, then the algorithm of \cite{LSW} estimates
the permanent with an error exponential in $n$.
In this case, $b_n=\Omega(n)$, and  Corollary \ref{cor: matrix perm conc} is
weaker than
Barvinok's theorem in \cite{Bar}. This is due to the fact that we do not have a good bound for the gap between $\E\log{\mbox{\rm det}^2(B_{1/2}\odot G)}$ and $\log \text{\rm per}(B)$, see \eqref{eq: perm-expectation-matrices}. However, this bound cannot be significantly improved in general, even for
well-connected matrices.
As we show in Lemma \ref{l: example}, the gap between these values is of order $\Omega(n)$ for a
matrix with all diagonal entries equal $1$ and all off-diagonal entries
equal $c/n$. For such
a matrix, the Barvinok--Godsil-Gutman estimator will fail consistently,
i.e., it will be concentrated around a value, which is $\exp(cn)$ far away from the permanent. Thus, we conclude that for almost doubly stochastic matrices with a broadly connected graph
the Barvinok--Godsil-Gutman  estimator either approximates the permanent up to  $\exp(o(n))$ with high probability, or  yields an exponentially big error with high probability.

As in \cite{FRZ}, Theorems \ref{cor-1} and \ref{cor-2}
depend on concentration of linear statistics
of the spectrum of random (inhomogeneous)
Gaussian matrices; this in turn require
a good control
on small singular values of such matrices. Thus, the first
part
of the current paper deals with the latter question, and proceeds as follows.
In Section \ref{sec-def} we define the notion of
broadly connected bipartite graphs, and state our main results
concerning small singular values of Gaussian matrices, Theorems
\ref{th: smallest singular value} and
\ref{th: intermediate singular}; we also state applications
of the latter theorems to both
adjacency graphs and to ``almost'' doubly stochastic matrices, see
Theorems \ref{th: graph matrix} and \ref{th: doubly stochastic}.
Section \ref{sec-net} is devoted to several preliminary lemmas involving $\e$-net
arguments. In Section \ref{sec-comp} we recall the notion of compressible
vectors and obtain estimate on the norm of Gaussian matrices restricted to
compressible vectors. The control of the minimal singular value
(that necessitates the study of incompressible vectors)
is obtained in Section \ref{sec-ssv}, while Section \ref{sec-isv} is devoted
to the study of intermediate singular values.
In Section \ref{sec-app}, we return to the analysis of the
Barvinok--Godsil-Gutman estimator, and use the control on singular
values together with an improved (compared to \cite{FRZ}) use of
concentration inequalities to prove the
applications and  the main theorems in the introduction.

\noindent
{\bf Acknowledgment}
We thank A. Barvinok and A. Samorodnitsky for sharing with us their knowledge
of permanent approximation algorithms, and for useful suggestions. We also
thank U. Feige for a useful suggestion.
\section{Definitions and results}
\label{sec-def}
For a matrix $A$ we denote its operator norm by $\norm{A}$,
and set $\norm{A}_{\infty}=\max |a_{i,j}|$.
By $[n]$ we denote the set $\{1 \etc n\}$.
By $\lfloor t \rfloor$ we denote the integer part of $t$.

 Let
$J \subset [m]$. Denote by $\R^J$ and $S^J$ the coordinate subspace of $\R^m$ corresponding to $J$ and its unit sphere.

For a left vertex $j \in [m]$ and a right vertex $i \in [n]$
of a bipartite graph $\Gamma=([m],[n], E)$
we write $j \to i$ if $j$ is connected to $i$.
\begin{definition}
  \label{def-scon}
 Let $\d ,\k>0, \ \d/2>\k$. Let $\Gamma$ be an $m \times n$ bipartite graph. We will say that $\Gamma$ is $(\d,\k)$-{\em broadly connected} if
 \begin{enumerate}
  \item $\text{deg}(i) \ge \d m$ for all $i \in [n]$;
  \item $\text{deg}(j) \ge \d n$ for all $j \in [m]$;
  \item for any set $J \subset [m]$  the set of its {\em broadly connected} neighbors
  \[
    I(J)=\{i \in [n] \mid j \to i \text{ for at least }  \lfloor (\d/2) \cdot |J| \rfloor \text{ numbers } j \in J \}
  \]
  has cardinality $|I(J)| \ge \min \big((1+\k) |J|, n \big)$.
 \end{enumerate}
\end{definition}
 We fix the numbers $\d, \k$ and call such graph  broadly connected.
Property (3) in this definition is similar to the expansion property of the graph.
In the argument below we denote by  $C,c$, etc.  constants depending on
the parameters
$\d, \k$ and $r$ appearing in Theorems
\ref{th: smallest singular value}
and
\ref{th: intermediate singular}.
The values of these constants may change from line to line.

Although condition (3) is formulated for all sets $J \subset [m]$,
it is enough to check it only for
sets with cardinality $|J| \le (1-\d/2)m$. Indeed, if $|J|> (1-\d/2)m$, then any $i \in [n]$ is broadly connected to $J$.
\begin{definition}
 Let $A$ be an $m \times n $ matrix. Define the graph $\Gamma_A=([m],[n], E)$ by setting $j \to i$ whenever $a_{j,i} \neq 0$.
\end{definition}

We will prove two theorems bounding the singular values of a matrix with normal entries.
In the theorems, we allow for non-centered entries because it will be useful
for the application of the theorem in the proof of Theorem
\ref{th: doubly stochastic}
\begin{theorem}  \label{th: smallest singular value}
 Let  $W$ be an $n \times n$ matrix with independent normal entries $w_{i,j} \sim N(b_{i,j}, a_{i,j}^2)$.
 Assume that
 \begin{enumerate}
   \item  $a_{i,j} \in \{0\} \cup [r, 1]$ for some constant $r>0$ and all $i,j$;
   \item  the graph $\Gamma_A$ is broadly connected;
   \item $\norm{\E W} \le K \sqrt{n}$ for some $K \ge 1$.
 \end{enumerate}
  Then for any $t>0$
  \[
    \P (s_n(W) \le c t K^{-C} n^{-1/2})
     \le t+ e^{-c' n}.
  \]
\end{theorem}

\begin{theorem}  \label{th: intermediate singular}
 Let $n/2< m \le n-4$,  and let $W$  be an $n \times m$ matrix with independent normal entries $w_{i,j} \sim N(b_{i,j}, a_{i,j}^2)$.
 Assume that
 \begin{enumerate}
   \item  $a_{i,j} \in \{0\} \cup [r, 1]$ for some constant $r>0$ and all $i,j$;
   \item  the graph $\Gamma_A$ is broadly connected;
   \item $\norm{\E W} \le K \sqrt{n}$.
 \end{enumerate}
  Then for any $t>0$
  \[
    \P \left(s_m(W) \le c t K^{-C} \cdot \frac{n-m}{\sqrt{n}} \right)
     \le t^{(n-m)/4}+ e^{-c' n}.
  \]
\end{theorem}
In Theorems~\ref{th: smallest singular value}, \ref{th: intermediate singular} we assume that the graph $\Gamma_A$ is broadly connected. This condition can be relaxed. In fact, property (3) in the definition of broad connectedness is used only for sets $J$ of cardinality $|J| \ge (r^2 \d/6) m$ (see Lemmas~\ref{l: sparse vector} and \ref{l: very sparse} for details).

We apply Theorems \ref{th: smallest singular value} and \ref{th: intermediate singular} to two types of matrices. Consider first the situation when the matrix $A$ is an adjacency matrix of a graph, and $\E W=0$.
\begin{theorem}  \label{th: graph matrix}
 Let $\G$ be a broadly connected $n \times n$ bipartite graph, and let $A$ be its adjacency matrix. Let $G$ be the $n \times n$ standard Gaussian matrix. Then for any $t>0$
  \[
    \P (s_n(A \odot G) \le c t  n^{-1/2})
     \le t+ e^{-c' n},
  \]
  and for any $n/2< m<n-4$
  \[
    \P \left(s_m(A \odot G) \le c t \cdot \frac{n-m}{\sqrt{n}} \right)
     \le t^{(n-m)/4}+ e^{-c' n}.
  \]
\end{theorem}
 Theorem \ref{th: graph matrix} is also applicable to the case when $\G$ is an unoriented graph with $n$ vertices. In this case we denote by $A$ its adjacency matrix, and assume that the graph $\G_A$ is broadly connected.
\begin{remark}
 With some additional effort the bound $m<n-4$
 in Theorem
 \ref{th: graph matrix}
 can be eliminated, and the term $ t^{(n-m)/4}$ in the right hand side can be replaced with $t^{n-m+1}$.
\end{remark}

The second application pertains to ``almost''
doubly stochastic matrices, i.e. matrices with uniformly bounded norms of
rows and columns.
\begin{theorem}  \label{th: doubly stochastic}
 Let $W$ be an $n \times n$ matrix with independent normal entries $w_{i,j} \sim N(0, a_{i,j}^2)$. Assume that the matrix of
 variances
$(a_{i,j}^2)_{i,j=1}^n$ satisfies the conditions
 \begin{enumerate}
   \item $ \sum_{i=1}^n a_{i,j}^2 \le C$ for any $j \in [n]$, and
   \item $ \sum_{j=1}^n a_{i,j}^2 \le C$ for any $i \in [n]$.
 \end{enumerate}
Consider an $n \times n$ bipartite graph $\G$ defined as follows:
 \[
  i \to j, \text{ whenever } \frac{c}{n} \le a_{i,j}^2,
 \]
  and assume that $\G$ is broadly connected.
 Then for any $t>0$
  \[
    \P (s_n(W) \le c t  n^{-1} \log^{-C'}n)
     \le t+ \exp(-C \log^4 n),
  \]
  and for any $n/2< m<n-4$
  \[
    \P \left(s_m(W) \le c t \cdot \frac{n-m}{n \log^{C'} n} \right)
     \le t^{(n-m)/4}+ \exp(-C \log^4 n).
  \]
\end{theorem}

Note that the condition on the
variance matrix  in Theorem \ref{th: doubly stochastic} does
not exclude the situation where
several of its entries  $a_{i,j}^2$ are of the order $\Om(1)$. Also, $\exp(-C \log^4 n)$ in the probability estimate can be replaced by $\exp(-C \log^p n)$ for any  $p$.
Of course, the constants
$C,C',c$ would then depend on $p$.

\section{Matrix norms and the $\e$-net argument}
\label{sec-net}

We prepare in this section some preliminary estimates that will be
useful in bounding probabilities by $\e$-net arguments. First,
we have the following
bound on
the norm of a random matrix as an operator acting between subspaces of $\R^n$.
This will be useful in the proof of
Theorem \ref{th: intermediate singular}.

\begin{lemma}  \label{l: subspace norm}
 Let $A$ be an $n \times n$ matrix with $\norm{A}_{\infty}
 \le 1$, and let $G$ be an $n \times n$ standard Gaussian matrix.
 Then for any subspaces $E,F \subset \R^n$ and any $s \ge 1$,
 \begin{multline*}
  \P (\norm{P_F(A \odot G): E \to \R^n}
  \ge c s(\sqrt{\dim E} + \sqrt{ \dim F}))
  \\
  \le \exp(-C s^2 (\dim E+ \dim F)),
\end{multline*}
 where $P_F$ is the orthogonal projection onto $F$.
\end{lemma}
\begin{proof}
 When $a_{i,j}\equiv 1$, the lemma is a direct consequence
 of the rotational invariance of the Gaussian measure,
 and standard concentration estimates for the top singular value of a Wishart
 matrix \cite[Proposition 2.3]{RV2}.
 For general $A$ satisfying the assumptions of the
 lemma,  the claim follows from the contraction argument
 in e.g. \cite[Lemma 2.7]{Sz}, since the collection of
 entries $\{g_{i,j}\}$ so that
  $\norm{A \odot G: E \to F} \le c s(\sqrt{\dim E} + \sqrt{ \dim F}))$
 is a convex symmetric set. We give
 an alternative direct proof:
 let $A'_{i,j}=\sqrt{1-A_{i,j}^2}$, and note that
 $G$ equals in distribution
 $A\odot G_1+A'\odot G_2$ where $G_1,G_2$ are independent copies of $G$.
 On the event
 $$\mathcal{A}_1:=
 \left\{\norm{P_F(A \odot G_1): E \to \R^n} \ge c s(\sqrt{\dim E} +
 \sqrt{ \dim F})\right\},$$
 there exist unit vectors $v_{G_1}\in F,w_{G_1}\in E$ so that
 $|v_{G_1}^T A \odot G_1 w_{G_1}|
  \ge c s(\sqrt{\dim E} + \sqrt{ \dim F})$.
  On the other hand, for any fixed $v,w$,
 $v^T A' \odot G_2 w$ is a Gaussian variable of variance
 bounded by $1$, and hence the event
 $$\mathcal{A}_2(v,w):=
 \left\{|v^T A' \odot G_2 w|
 \ge c s(\sqrt{\dim E} + \sqrt{ \dim F})/2\right\}$$
 has probability bounded above by
$$\exp(- C s^2(\sqrt{\dim E} + \sqrt{ \dim F})^2)
  \le \exp(-C s^2 (\dim E+ \dim F)).$$
The proof is completed by noting that
\begin{eqnarray*}
  \P(\mathcal{A}_1)&\leq&
\E\P(\mathcal{A}_2(v_{G_1},w_{G_1})\mid \mathcal{A}_1))
  \\&&
  \quad
  + \P (\norm{ P_F G: E \to \R^n}
  \ge c s(\sqrt{\dim E} + \sqrt{ \dim F})/2)\,.
\end{eqnarray*}
%
\end{proof}

To prove Theorem \ref{th: doubly stochastic} we will
need an estimate of the norm of the matrix, which is based on a result of Riemer and Sch\"{u}tt \cite{RS}.

\begin{lemma} \label{l: riemer-schutt}
 Let $A$ be an $n \times n$ matrix
 satisfying conditions (1) and (2) in Theorem \ref{th: doubly stochastic}.
 Then
 \[
  \P(\norm{A \odot G} \ge C \log^2 n) \le \exp(-C \log^4 n).
 \]
\end{lemma}
\begin{proof}
  Write $X=A\odot G$.
 By \cite[Theorem 1.2]{RS},
 \begin{equation}
   \label{RS1}
 \E \norm{A \odot G} \le C (\log^{3/2} n) \E(\max_{i=1,\ldots,n}
 \|(X_{i,j})_{j=1}^n\|_2+
 \|(X_{i,j})_{i=1}^n\|_2).
\end{equation}
 Set $\eta_i=
 \|(X_{i,j})_{j=1}^n\|_2$, $i=1,\ldots,n$ and $\Delta_i=\sum_{j=1}^n a_{i,j}^2
 \leq C$.
 Define $\beta_{i,j}=a_{i,j}^2/\Delta_i\leq 1$.
 For $\theta\leq 1/4C$ one has that
 $$\log \E e^{\theta \eta_i^2}
 =-\frac12\sum_{j=1}^n\log(1-2\beta_{i,j}\theta \Delta_i)
 \leq c\theta\,,$$
 for some constant $c$ depending only on $C$. In particular,
 the independent random variables
 $\eta_i$ possess uniform (in $i,\theta,n$) subgaussian tails, and
 therefore,
 $\E\max_{i=1,\ldots,n} \eta_i\leq c' (\log n)^{1/2}$.
 Arguing similarly for $\E (\max_{i=1,\ldots,n}
 \|(X_{i,j})_{i=1}^n\|_2))$ and substituting in
 \eqref{RS1}, one concludes that
 \[\E \norm{A \odot G} \le C \log^2 n.\]
 The lemma follows from the concentration for the Gaussian measure, since $F: \R^{n^2} \to \R, \ F(B)=\norm{A \odot B}$ is a 1-Lipschitz function, see
 e.g. \cite{Ledoux}.
\end{proof}

Throughout the proofs below we will repeatedly use the easiest form of the $\e$-net argument. For convenience, we  will formulate it as a separate lemma.
\begin{lemma} \label{l: epsilon-net}
 Let $V$ be a $n \times m$  random matrix. Let $\LL \subset S^{m-1}$ be a set contained in an $l$-dimensional subspace of $\R^m$.  Assume that there exists $\e>0$ such that for any $x \in \LL$
 \[
   \P( \norm{Vx}_2 <\e \sqrt{n}) \le p.
 \]
 Denote by $\LL_\a$ the $\a$-neighborhood of $\LL$
 in $\R^m$.
 Then
 \[
  \P(\exists x \in \LL_{\e/(4K)}: \ \norm{Vx}<(\e/2) \cdot \sqrt{n} \text{ and } \norm{V} \le K \sqrt{n})
  \le \left( \frac{6 K}{\e} \right)^l \cdot p.
 \]
\end{lemma}
\begin{proof}
 Let $\NN \subset \LL$ be an $(\e/(4K))$-net in $\LL$. By the volumetric estimate, we can choose $\NN$ of cardinality
 \[
  |\NN| \le \left(\frac{6K}{\e} \right)^l.
 \]
 Assume that there exists $y \in \LL_{\e/(4K)}$ such that $\norm{Vy}_2 <(\e/2) \cdot \sqrt{n}$. Choose $x \in \NN$ for which $\norm{y-x}_2<\e/(2K)$. If $\norm {V} \le K \sqrt{n}$, then
 \[
   \norm{Vx}_2 \le (\e/2) \sqrt{n} +\norm{V}\cdot \frac{\e}{2K} \le \e \sqrt{n}.
 \]
 Therefore, by the union bound,
 \begin{multline*}
   \P(\exists y \in \LL_{\e/(4K)}: \ \norm{V y}<(\e/2)\sqrt{n}
   \text{ and } \norm{V} \le K \sqrt{n}) \\
  \le \P(\exists x \in \NN : \ \norm{Vx}\leq
  \e\sqrt{n})
  \le \left( \frac{6 K}{\e} \right)^l \cdot p.
 \end{multline*}
\end{proof}

\section{Compressible vectors}
\label{sec-comp}

As developed in detail in \cite{RV,RV2}, when estimating singular values it
is necessary to handle separately the action of random matrices
on compressible, e.g., close to sparse, vectors. We begin with
a basic small ball estimate.
\begin{lemma} \label{l: sparse vector}
Let $m, n \in \N$.
 Let $A, B$ be  (possibly random)
 $n \times m$ matrices, and let $W=A \odot G+B$, where $G$
is the $n \times m$ Gaussian matrix, independent of $A,B$.
 Assume that, a.s.,
 \begin{enumerate}
   \item  $a_{i,j} \in \{0\} \cup [r, 1]$ for some constant $r>0$ and all $i,j$;
   \item  the graph $\Gamma_A$ satisfies  $\text{deg}(j) \ge \d n$ for all $j \in [m]$.
 \end{enumerate}
 Then
  for any $x \in S^{m-1}$, $z \in \R^n$ and for any $t>0$
 \[
   \P(\norm{Wx-z}_2 \le t \sqrt{n}) \le (Ct)^{ c n}.
 \]
\end{lemma}

\begin{proof}
 Let $x \in S^{m-1}$. Set $I=\{i \in [n] \mid \sum_{j=1}^m a_{i,j}^2 x_j^2 \ge  r^2 \d/2\}$.
 Let $\Gamma=\Gamma_{A^T}$ be the graph of the matrix $A^T$. The inequality
 \[
  \sum_{i=1}^n \sum_{j=1}^m  a_{i,j}^2 x_j^2
  \ge \sum_{j=1}^m r^2 \text{deg}_{\Gamma}(j) x_j^2
  \ge  r^2 \d n \sum_{j=1}^m x_j^2= r^2 \d n
 \]
 implies
 \[
   \sum_{i \in I} \left( \sum_{j=1}^m   a_{i,j}^2 x_j^2 \right)
   \ge \sum_{i=1}^n \sum_{j=1}^m  a_{i,j}^2 x_j^2 -r^2 \d n/2 \ge r^2 \d n/2.
 \]
On the other hand, we have the reverse inequality
 \[
   \sum_{i \in I} \left( \sum_{j=1}^m   a_{i,j}^2 x_j^2 \right)
    \le    |I| \left( \sum_{j=1}^m    x_j^2 \right)= |I|,
 \]
 and so $|I| \ge r^2\d n/2$.

 For any $i \in I$ the
independent
normal random variables $w_i=\sum_{j=1}^m (a_{i,j} g_{i,j}+b_{i,j}) x_j$
have variances at least $r^2 \d/2$.
Estimating the Gaussian measure of a ball by its
Lebesgue measure, we get that for any $\t>0$
 \begin{multline*}
   \P \left(\norm{Wx-z}_2^2 \le \t^2 (r^2 \d/2)^2 \cdot  n\right) \\
   \le  \P \left( \sum_{i \in I} (w_i-z_i)^2  \le  \t^2 (r^2 \d/2) \cdot |I| \right)
   \le (C \t)^{|I|}.
 \end{multline*}
 Setting $t= \t r^2 \d /2$ finishes the proof.
\end{proof}

We now introduce the notion of compressible and incompressible vectors.
The compressible vectors will be easier to handle by an $\e$-net argument,
keeping
track of the
degree of compressibility. This is the content of the next three
lemmas in this section.

 For $u, v<1$ denote
 \[
  \text{Sparse}(u) = \{ x \in S^{m-1} \mid |\supp(x)| \le u m\}.
 \]
 and
  \begin{align*}
    \text{Comp}(u,v) &= \{ x \in S^{m-1} \mid \exists y \in \text{Sparse}(u), \ \norm{x-y}_2 \le v \}, \\
     \text{Incomp}(u,v) &= S^{m-1} \setminus  \text{Comp}(u,v).
  \end{align*}
 We employ the following strategy. In Lemma \ref{l: very sparse}, we show that the matrix $W$ is well invertible on the set of highly compressible vectors. Lemma \ref{l: intermediate compressible} asserts that if the matrix is well invertible on the set of vectors with a certain degree of compressibility, then we can relax the compressibility assumption and show invertibility on a larger set of compressible vectors. Finally, in Lemma \ref{l: all compressible}, we  prove that the matrix $W$ is well invertible on the set of all compressible vectors. This is done by using Lemma \ref{l: very sparse} for highly compressible vectors, and extending the set of vectors using Lemma \ref{l: intermediate compressible} in finitely many steps. The number of these steps will be independent of the dimension.
\begin{lemma} \label{l: very sparse}
Let $m, n \in \N, \ m \le (3/2)n$.
 Let $A, B, W$ be  $n \times m$ matrices satisfying
the conditions of Lemma \ref{l: sparse vector}. Let $K \ge 1$.
Then there exist constants $c_0,c_1,c_2$ such that,
 for any  $z \in \R^n$,
  \begin{multline*}
    \P \Big( \exists x \in \text{\rm Comp}\left(c_0, c_1/K^2 \right):\\
    \ \norm{Wx-z}_2 \le  (c_1/K) \sqrt{n}
    \text{ and } \norm{W} \le K \sqrt{n} \Big)
    \le e^{- c_2 n}.
  \end{multline*}
\end{lemma}

\begin{proof}
  Let $c$ be the constant from Lemma \ref{l: sparse vector}.
Without
loss of generality, we may and will assume
that $c<1$.
    Let $ t>0$ be a number to be chosen later.
      For any set $J \subset [m]$ of cardinality
$|J| =l=\lfloor cm/3 \rfloor$ Lemmas \ref{l: sparse vector}
and \ref{l: epsilon-net} imply
\begin{eqnarray*}
  &&\P(\exists x \in (S^{J})_{t/(4K)}: \
  \norm{Wx}_2<(t/2) \sqrt{n}  \text{ and } \norm{W} \le K \sqrt{n}) \\
  && \quad \quad \le\left( \frac{6 K}{t} \right)^l \cdot (Ct)^{cn}.
\end{eqnarray*}
(Recall that $S^J$ is
the unit sphere of
the coordinate subspace of $\R^m$ corresponding to $J$.)
Since $\text{Comp}(c/3, t/(4K)) \subset \bigcup_{|J|=l} (S^{J})_{t/(4K)}$, the
union bound
yields
  \begin{multline*}
  \P(\exists x \in \text{Comp}(c/3, t/(4K)):\ \norm{Wx}<(t/2) \sqrt{n} \text{ and } \norm{W} \le K \sqrt{n}) \\
  \le  \binom{m}{l} \cdot  \left( \frac{6 K}{t} \right)^l \cdot (Ct)^{cn}
  \le  \left( \frac{C K }{ t } \right)^{c m /3} \cdot (Ct)^{ c n},
 \end{multline*}
 which does
 not
 exceed $e^{-c n/3}$ provided
 that $t = c''/K$ for an appropriately chosen $c''>0$.
  This proves the lemma if we set $c_0=c/3, \ c_1=c''/4$.
\end{proof}

\begin{lemma} \label{l: intermediate compressible}
Let $m, n \in \N, \  n \le 2m$.
 Let $A, B$ be
 (possibly random) $n \times m$ matrices, and set $W= A \odot G+B$,
where $G$ is the standard $n \times m$ Gaussian matrix, independent
of $A,B$. Assume that, a.s.,
 \begin{enumerate}
   \item  $a_{i,j} \in \{0\} \cup [r, 1]$ for some constant $r>0$ and all $i,j$;
   \item  the graph $\Gamma_{A^T}$ is broadly connected.
 \end{enumerate}
  Then for any $c_0$ and any
$u,v>0$, such that $u \ge c_0$ and $(1+\k/2)u<1$, and for any $z \in \R^n$
  \begin{align*}
    &\P \left( \exists x \in \text{\rm Comp} ((1+\k/2)u, (v/K)^{C+1})
\setminus \text{\rm Comp}(u,v)
      \right.:
      \\
      &\quad \left.  \norm{Wx-z}_2 \le cv (v/K)^{C} \sqrt{n}\text{ and } \norm{W} \le K \sqrt{n} \right) \\
    &\le e^{-c  n}.
  \end{align*}
where $c=c(c_0,\kappa,\delta,r)$.
\end{lemma}

\begin{proof}
 Let $S(u,v)=\text{Sparse}((1+\k/2)u) \setminus \text{Comp}(u,v)$.
 Fix
 any $x \in \R^n$ and
 denote by $J$ the set of all coordinates $j \in [m]$
 such that $|x_j| \ge v/ \sqrt{m}$.
 For any $x \in S(u,v) \ |J| \ge um$, since otherwise $x \in \text{Comp}(u,v)$.
 Since the graph $\Gamma_{A^T}$ is broadly connected, this implies that $|I(J)| \ge (1+\k) u m$.

 If $i \in I(J)$, then $w_i=\sum_{j=1}^m a_{i,j} g_{i,j} x_j$ is a centered normal random variable with variance
 \[
   \s_i^2=\sum_{j=1}^m a_{i,j}^2 x_j^2
   \ge \frac{v^2}{ m} \cdot \sum_{j \in J} a_{i,j}^2
   \ge \frac{v^2}{m} \cdot r^2 (\d/2) |J|
   \ge v^2 r^2 u \d/2.
 \]
 Hence,  for any $t>0$,
 \begin{align*}
  &\P \left( \norm{Wx-z}_2 \le t  v r u \cdot \sqrt{\d  n} \right)
  \le \P \left( \norm{Wx-z}_2 \le t v r u \cdot \sqrt{\d  m/2} \right)
  \\
  &\le \P \left( \sum_{i \in I(J)} (w_i-z_i)^2
                \le t^2 v^2 r^2 u (\d/2) \cdot |I(J)| \right)
  \le (c t)^{|I(J)|}
  \le (c t)^{(1+\k)u m},
 \end{align*}
where the third inequality is obtained by the same reasoning as
at the end of the proof of Lemma \ref{l: sparse vector}.
 Let $\D \subset [m]$ be any set of cardinality $l=\lfloor (1+\k/2)u m \rfloor$, and denote $\Phi^\D=S^\D \cap S(u,v)$.
 Set $\e=tv r u \cdot \sqrt{\d }$.
  By Lemma \ref{l: epsilon-net},
 \begin{multline*}
    \P \left( \exists x \in (\Phi^\D)_{\e/(4K)}:
    \norm{Wx-z}_2 \le
    t \frac{v r u\sqrt{\d  n} }{2}\text{ and } \norm{W} \le K \sqrt{n} \right) \\
    \le (c t)^{(1+\k)u m} \cdot \left(\frac{6 K}{\e} \right)^l
 \end{multline*}
 We have
 \[
    \text{Comp} \left( \left(1+\frac{\k}{2} \right)u,  \frac{\e}{4K} \right) \setminus \text{Comp}(u,v) \subset \bigcup_{|\D|=l} (\Phi^\D)_{\e/(4K)}.
 \]
 Therefore, the union bound yields
 \begin{multline*}
    \P \left( \exists x \in  \text{Comp} \left( \left(1+\frac{\k}{2} \right)u,
    \frac{tv r u \cdot \sqrt{\d }}{4K} \right) \setminus \text{Comp}(u,v)
    \right.: \\
        \left.
	\norm{Wx-z}_2 \le  t \frac{v r u\sqrt{\d n} }{2}\text{ and } \norm{W} \le K \sqrt{n} \right) \\
    \le \binom{m}{l} \cdot(c t)^{(1+\k)u m} \cdot \left(\frac{6 K}{\e} \right)^l
    \le (c t)^{(1+\k)u m} \cdot \left(\frac{C K}{u^2 \cdot tv
    r \cdot \sqrt{\d }} \right)^{(1+\k/2)u m} \\
      \le \left[ \left(\frac{C' K}{ v } \right)^{4/\k} t \right]^{\k u m/2}.
 \end{multline*}
 This does  not exceed $e^{-\k u m/2}$ if we choose
 \[
  t= e^{-1} \cdot \left(\frac{C' K}{ v } \right)^{-4/\k}.
 \]
 Substituting this $t$ into the estimate above proves the lemma.
\end{proof}

\begin{lemma} \label{l: all compressible}
Let $m, n \in \N, \  (2/3)m \le n \le 2m$.
 Let $A, B$ be an $n \times m$ matrices, and set
$W= A \odot G+B$, where
$G$ is the standard $n \times m$ Gaussian matrix, independent of
$A,B$. Assume that
 \begin{enumerate}
   \item  $a_{i,j} \in \{0\} \cup [r, 1]$ for some constant $r>0$ and all $i,j$;
   \item  the graph $\Gamma_{A^T}$ is broadly connected.
 \end{enumerate}
  Then  for all $z \in \R^n$
   \begin{multline*}
    \P \Big( \exists x \in \text{\rm Comp} (1-\k/2, K^{-C}):\\
\quad\quad
   \norm{Wx-z}_2 \le K^{-C} \sqrt{n} \text{ and } \norm{W}\le K \sqrt{n}  \Big)
    \le e^{- c n}.
  \end{multline*}
\end{lemma}
\begin{proof}
 Set $u_0= c_0, \ v_0= c_1 K^{-2}$, where $c_0,c_1$ are the constants from Lemma \ref{l: very sparse}.
  Let $L$ be the smallest natural number such that
  \[
    u_0 (1+\k/2)^L>1-\k/2.
  \]
  Note that $u_0 (1+\k/2)^L \le (1-\k/2)\cdot (1+\k/2)<1$.
  Define by induction $v_{l+1}=(v_l/K)^{C+1}$, where $C$ is the  constant from Lemma \ref{l: intermediate compressible}.
  Then $v_L=K^{-C'}$ for some $C'>0$ depending only on the parameters $\d, \k$ and $r$.
  We have
 \begin{multline*}
  \text{Comp} (1-\k/2, v_L)
  \subset
  \text{Comp} (u_0, v_0) \ \cup
  \\
  \bigcup_{l=1}^{L} \text{Comp} (u_0(1+ \k/2)^l, v_l) \setminus
  \text{Comp} (u_0(1+ \k/2)^{l-1}, v_{l-1}).
 \end{multline*}
 The result now follows from Lemmas \ref{l: very sparse} and \ref{l: intermediate compressible}.
\end{proof}

\section{Smallest singular value}
\label{sec-ssv}

To estimate the smallest singular value, we need the following result
from \cite[Lemma 3.5]{RV}, that handles incompressible vectors.
\begin{lemma}  \label{l: via distance}
  Let $W$ be an $n \times n$ random matrix.
  Let $W_1,\ldots,W_n$ denote the column vectors of $W$, and let $H_k$ denote
  the span of all column vectors except the $k$-th.
  Then for every $a,b \in (0,1)$ and every $t> 0$, one has
  \begin{equation}                              \label{eq: via distance}
    \P \big( \inf_{x \in \text{\rm Incomp}(a,b)} \|W x\|_2 < t b n^{-1/2} \big)
    \le \frac{1}{a n} \sum_{k=1}^n
          \P \big( \dist( W_k, H_k) < t \big).
  \end{equation}
\end{lemma}

Now we can derive the first main result.

\begin{proof}[Proof of Theorem \ref{th: smallest singular value}]
Set $B= \E W$ and let $A=(a_{i,j})$, where $a_{i,j}^2= \text{Var}( w_{i,j})$, so
\[
  W=A \odot G +B,
\]
where $G$ is the $n \times n$ standard Gaussian matrix.

 Without loss of generality, assume that $K>K_0$, where $K_0>1$ is a constant to be determined.
  Applying Lemma \ref{l: very sparse} to the matrix $W$, we obtain
  \begin{multline*}
    \P \Big( \exists x \in \text{\rm Comp}
\left(c_0, c_1 K^{-2} \right):\\
 \ \norm{Wx}_2 \le (4c_1/K) \sqrt{n}
    \text{ and } \norm{W} \le K \sqrt{n} \Big)
    \le e^{- c n}.
  \end{multline*}
  Therefore, for any $t>0$
  \begin{multline*}
   \P(s_n(W) \le c t K^{-C} n^{-1/2}) \le e^{-cn} + \P (\norm{W} \ge K \sqrt{n}) \\
   +     \P ( \exists x \in \text{\rm Incomp}\
\left(c_0, c_1 K^{-2} \right): \ \norm{Wx}_2 \le (4c_1/K) \sqrt{n}).
  \end{multline*}
    By Lemma \ref{l: subspace norm},
  \[
    \P (\norm{W}> 2K \sqrt{n}) \le \P (\norm{A \odot G}> K \sqrt{n})
    \le e^{-cn},
  \]
  provided that $K>K_0$ with $K_0$ taken large enough, thus determining
$K_0$.
  By Lemma \ref{l: via distance}, it is enough to bound
  $ \P \big( \dist( W_k, H_k) < c t \big)$ for all $k \in [n]$.
Consider, for example, $k=1$. In the discussion that follows,
let $h \in S^{n-1}$
be a vector such that $h^T W_j=0$ for all $j=2 \etc n$. Then
  \[
    \dist( W_1, H_1) \ge |h^T W_1|.
  \]
  Let $\tilde{A}$ be the $(n-1) \times n$ matrix whose rows are the columns of $A^T$, except the first one, i.e. $\tilde{A}^T=(A_2, A_3 \etc A_n)$. Define the $(n-1) \times n$ matrices $\tilde{B}, \tilde{W}$ in the same way.
  The condition on $h$ can now be rephrased as $\tilde{W}h=0$.

Since the graph
$\Gamma_{A}$ is broadly connected,
the graph $\Gamma_{\tilde{A}^T}$
is broadly connected with slightly smaller parameters and in particular
with
parameters $\d/2$ and $\k/2$.
Since
    $ \text{Comp}(1-\k/2, (2K)^{-C})\subset
     \text{Comp}(1-\k/4, (2K)^{-C})$, we get from
  Lemma \ref{l: all compressible} applied to
$\tilde{W}$, $z=0$, and with $K$ replaced by $2K$, that
  \begin{multline*}
    \P\Big(\exists h \in \text{\rm Comp}(1-\k/2, (2K)^{-C}), \ \tilde{W}h=0\Big)\\
    \le \P\Big(
\exists h \in \text{\rm Comp}(1-\k/2, (2K)^{-C}),
\norm{\tilde{W}h}_2 \le (2K)^{-C'} \sqrt{n}\\
 \text{ and } \norm{\tilde{W}} \le 2K \sqrt{n}\Big)
     + \P (\norm{\tilde{W}}> 2K \sqrt{n})\\
     \le e^{-cn}+
      \P (\norm{\tilde{W}}> 2K \sqrt{n})\, .
  \end{multline*}
  The last term is exponentially small:
  \[
    \P (\norm{\tilde{W}}> 2K \sqrt{n}) \le  \P (\norm{W}> 2K \sqrt{n})
    \le e^{-cn}
  \]
 Hence,
  \[
    \P\Big( \exists
 h \in \text{Comp}(1-\k/2, (2K)^{-C}): \ \tilde{W}h=0\Big)
\le e^{-c'n}.
  \]
  Note that the vector $h$ is independent of $W_1$. Therefore,
  \begin{align*}
   &\P \big( \dist( W_1, H_1) <  t (2K)^{-C}  \big) \\
   &\le \P(|h^T W_1| \le  t (2K)^{-C}, \ \tilde{W}h=0, \text{ and } h \in \text{Comp}(1-\k/2, (2K)^{-C})
   \\
  &\quad  + \P(|h^T W_1| \le  t (2K)^{-C}, \ \tilde{W}h=0, \text{ and } h \notin \text{Comp}(1-\k/2, (2K)^{-C})
  \\
  &\le e^{-c'n}
    +\E \P_{W_1}(|h^T W_1| \le  t (2K)^{-C}
\mid  h \notin \text{Comp}(1-\k/2, (2K)^{-C})
  \\
  &\le e^{-c'n}
    + \sup_{u \in \text{Incomp}(1-\k/2, (2K)^{-C})} \P(|u^T W_1| \le c t K^{-C})
  \end{align*}
  Assume that $u \in \text{Incomp}(1-\k/2, (2K)^{-C})$. Let $J= \{j \in [n]
: |u_j| \ge (2K)^{-C} n^{-1/2} \}$. Then $|J| \ge (1-\k/2)n$.
  Hence, if $J'=\{j \in [n] :
|a_{1j}| \geq r n^{-1/2}$, then $|J \cap J'|\ge (\d-\k/2)n>(\d/2)n$.
  Therefore,
   $u^T W_1$ is a centered normal random variable with variance
$\s^2 \ge r^2 (2K)^{-2C}\cdot \d/2$, and so
  \[
    \P(|u^T W_1| \le  t (2K)^{-C}) \le C' t.
  \]
  This means that
  \[
       \P \big( \dist( W_1, H_1) <  t (2K)^{-C} \big)
       \le  t+ e^{-c n},
  \]
  and the same estimate holds for $\dist( W_j, H_j), \ j>1$, so the theorem follows from Lemma \ref{l: via distance}.

\end{proof}

\section{Intermediate singular value}
\label{sec-isv}
 The next elementary lemma allows one to find a set of rows of a fixed matrix with big $\ell_2$ norms, provided that the graph of the matrix has a large minimal degree.
\begin{lemma}  \label{l: dense subset}
 Let $k<n$, and let $A$ be an $n \times n$ matrix.
  Assume that
 \begin{enumerate}
   \item  $a_{i,j} \in \{0\} \cup [r, 1]$ for some constant $r>0$ and all $i,j$;
   \item  the graph $\Gamma_A$ satisfies  $\text{deg}(j) \ge \d n$ for all $j \in [n]$.
 \end{enumerate}
  Then for any $J \subset [n]$ there exists a set $I \subset [n]$ of cardinality
 \[
  |I| \ge (r^2 \d/2) n,
 \]
 such that for any $i \in I$
 \[
  \sum_{j \in J} a_{i,j}^2 \ge (r^2 \d/2) \cdot |J|.
 \]
\end{lemma}

\begin{proof}
 By the assumption on $A$,
 \[
  \sum_{i=1}^n \sum_{j \in J} a_{i,j}^2 \ge r^2 \d n \cdot |J|.
 \]
 Let $I=\{i \in [n] \mid \sum_{j \in J} a_{i,j}^2 \ge r^2 \d  \cdot |J|/2 \}$. Then
 \begin{align*}
  |J| \cdot |I|
  \ge \sum_{i \in I} \sum_{j \in J} a_{i,j}^2
  \ge r^2 \d n \cdot |J| - \sum_{i \in I^c} \sum_{j \in J} a_{i,j}^2
  \ge \frac{r^2 \d |J|}{2} \cdot n.
 \end{align*}
\end{proof}

 We also need the following lemma concerning the Gaussian measure in $\R^n$.
 \begin{lemma}  \label{l: projection measure}
  Let $E,F$ be
linear
subspaces of $\R^n$. Let $P_E, P_F$ be the orthogonal projections
onto $E$ and $F$, and assume that
for some $\tau>0$,
  \[
   \forall y \in F, \ \norm{P_E y}_2 \ge \t \norm{y}_2.
  \]
  Let $g_E$ be the standard Gaussian vector in $E$. Then for any $t>0$
  \[
   \P( \norm{P_F g_E}_2 \le t)
   \le \left( \frac{ct}{\t \sqrt{\dim F}} \right)^{\dim F}.
  \]
 \end{lemma}

 \begin{proof}
  Let $E_1= P_E F$. Then (because $\tau>0$),
the linear operator $P_E: F \to E_1$ has a trivial kernel and hence
is a bijection.
  Denote by $g_H$ the standard Gaussian vector in the space $H \subset \R^n$.
  Let $U: \R^n \to \R^n$ be an isometry such that
  $U E_1=F$ and $UF=E_1$. Then  $P_F=U P_{E_1} U$ and $Ug_{E_1}$ has the same distribution as
$
  g_F$.
  Therefore, integrating over the coordinates of $g_E$ orthogonal to $E_1$, we get
  \begin{multline*}
   \P( \norm{P_F g_E}_2 \le t) \le    \P(  \norm{U P_{E_1} U g_{E_1}}_2 \le t) \\
   =    \P(\norm{P_{E_1} g_F}_2 \le t)
   \le \P(\norm{g_F}_2 \le t/\t).
  \end{multline*}
  The lemma follows from the standard density estimate for the Gaussian vector.

 \end{proof}

 Let $J \subset [m]$.
 For levels $Q>q > 0$
define the set of totally spread vectors
\begin{equation}                        \label{SJ}
  \mathcal{S}^J_{q,Q}:= \Big\{ y \in
S^J: \;
    \frac{q}{\sqrt{|J|}} \le |y_k| \le \frac{Q}{\sqrt{|J|}}
    \quad \text{for all $k \in J$} \Big\}.
\end{equation}

\begin{lemma}                     \label{l: via dist multidim}
  Let $\delta,\rho \in (0,1)$.
  There exist $Q>q > 0$ and $\a,\b>0$, which depend polynomially on $\d,\rho$,
   such that the following holds.
   Let $d \le m \le n$ and let $W$ be an $n \times m$ random matrix with independent columns.
   For $I \subset [m]$ denote by $H_I$ the linear subspace spanned by the columns $W_i, \ i \in I$.
   Let $J$ be a uniformly chosen random subset of $[n]$ of cardinality $d$.
  Then for every $\e > 0$
  \begin{multline}                                  \label{eq via dist}
    \P \Big( \inf_{x \in \text{\rm Incomp}(\delta,\rho)} \|Wx\|_2 < \a \e \sqrt{\frac{d}{n}} \Big) \\
    \le  \b^d  \cdot \E_J  \P \big( \inf_{z \in \mathcal{S}_{q,Q}^J} \dist(W z, H_{J^c}) < \e \big).
  \end{multline}
\end{lemma}
 \begin{remark}
 Lemma \ref{l: via dist multidim} was proved in \cite{RV2} for
 random matrices with i.i.d. entries (see Lemma 6.2 there).
 However, that
 proof can be extended to the general case without any changes.
 \end{remark}

\begin{proof}[Proof of Theorem \ref{th: intermediate singular}]
Set $B= \E W$ and let $A=(a_{i,j})$, where $a_{i,j}^2= \text{Var}( w_{i,j})$, so
\[
  W=A \odot G +B,
\]
where $G$ is the $n \times n$ standard Gaussian matrix.
 Without loss of generality assume that
\begin{equation}\label{eq: kappa small}
    \k \le \frac{r^2 \d}{2}.
\end{equation}
If this inequality doesn't hold, we can redefine $\k$ as the right hand side of this inequality, and note that the broad connectedness property is retained when $\k$ gets smaller.

  Let $C>0$ be as in Lemma \ref{l: all compressible}.
  Decomposing the sphere into compressible and incompressible vectors,
  we write
  \begin{eqnarray}
    \label{eq-o1}
    \nonumber
    &&\P \left(s_m(W) \le c t K^{-C} \cdot \frac{n-m}{\sqrt{n}} \right) \\
    &\le& \P \left(\inf_{x \in \text{Comp}(c_0, c_1 K^{-2})} \norm{Wx}_2 \le c t K^{-C} \cdot \sqrt{n} \right) \\
    &&+ \P \left(\inf_{x \in \text{Incomp}(c_0, c_1 K^{-2})} \norm{Wx}_2 \le c t K^{-C} \cdot \frac{n-m}{\sqrt{n}} \right).
    \nonumber
  \end{eqnarray}
  By Lemma \ref{l: very sparse}, the first term in the right side of
  \eqref{eq-o1} does not exceed
  \[
   e^{-c_2 n}+ \P(\norm{W} \ge 2 K \sqrt{n}).
  \]
  By Lemma
\ref{l: subspace norm},
the last term in the last expression
is smaller than $e^{-cn}$, if $K$ is large enough.

  To estimate the second term in the right side of
  \eqref{eq-o1} we use Lemma \ref{l: via dist multidim}.
  Recall that by that
 lemma, we can assume that $q=K^{-C'}$ and $Q=K^{C'}$ for some constant $C'$. Then
the
lemma reduces the problem to estimating
  \[
    \P \big( \inf_{z \in \mathcal{S}_{q,Q}^J} \dist(W z, H_{J^c}) < \e \big)
  \]
  for these $q,Q$ and for
  a fixed subset $J \subset [m]$ of cardinality
  \[
    d= \left \lfloor \frac{n-m}{2} \right \rfloor,
  \]
  and with a properly chosen $\e$, see \eqref{eq: def epsilon} below.

    Since we do not control the norm of the submatrix matrix $B$ corresponding to $J$, we will reduce the dimension further to eliminate this matrix.
  Set $H_0= B \R^J \subset \R^n$, and let $F= (H_{J^c} \cup H_0)^{\perp}$.
  Then $F$ is a linear subspace of $\R^n$ independent of
  $\{W_j, \ j \in J\}$, and
  \begin{equation}
    \label{newstar}
    n-m\leq \dim F\leq n-m+d\leq 2(n-m)\,.
  \end{equation}
    Since $P_F  B \R^J=\{0\}$, we get
  \begin{align}\label{eq: dist F}
       & \P \big( \exists z \in \mathcal{S}_{q,Q}^J :\  \dist(W z, H_{J^c}) < \e \big)
    \le \P \big( \exists z \in \mathcal{S}_{q,Q}^J :\  \norm{P_F W z}_2 < \e \big) \\
    &=  \P \big( \exists z \in \mathcal{S}_{q,Q}^J: \   \norm{P_F (A \odot G) z}_2 < \e \big)  \notag
  \end{align}
  for any $\e>0$.

  We start with bounding the small ball probability for a fixed vector
$z \in \mathcal{S}_{q,Q}^J$. The $i$-th coordinate of the vector $(A \odot G) z$ is a normal random variable with variance
  \[
   \s_i^2=\sum_{j \in J}a_{i,j}^2 x_j^2 \ge \frac{q^2}{d}\sum_{j \in J}a_{i,j}^2.
  \]
  Let $I \subset [n]$ be the set constructed in Lemma \ref{l: dense subset}. Then for any $i \in I$ we have $\s_i \ge c q= c'K^{-C'}$.
  Let $E$ be the subspace of $\R^n$ spanned by the vectors $e_i, \ i \in I$.

  Since $P_E (A \odot G)z$ and $P_{E^\perp} (A \odot G)z$ are independent Gaussian vectors,
  \begin{align}  \label{eq: projection-norm}
    &\P \big(  \norm{P_F (A \odot G) z}_2 < \e \big)  \\
   & = \E_{ P_{E^\perp} (A \odot G)} \P \big(  \norm{P_F P_E (A \odot G) z + P_F P_{E^\perp} (A \odot G) z}_2 < \e \mid  P_{E^\perp} (A \odot G) \big) \notag \\
    &\le \P \big(  \norm{P_F P_E (A \odot G) z}_2 < \e \big)
    \le \P \big(   \norm{P_F g_E}_2 < c K^{C'}\e \big). \notag
  \end{align}
  Here $g_E$ is the standard Gaussian vector in $E$.
  The first
inequality in \eqref{eq: projection-norm}
is a consequence of Anderson's
inequality \cite[Theorem 1]{A55}, applied to the convex symmetric
function $f(x)={\bf 1}_{\|x\|_2<\e}$ and the Gaussian
random vector $P_FP_E(A\odot G)x$.
The last inequality in \eqref{eq: projection-norm}
follows since $P_E (A \odot G) z$ is a
 vector with independent normal coordinates with variances greater than
  $c' K^{-C'}$.

  Now we have to check that the spaces $E$ and $F$
satisfy the conditions of Lemma \ref{l: projection measure},
with high probability.
 Let $\tilde{A}, \tilde{B}, \tilde{G}$, and $\tilde{W}$ be $(m-d)\times n$ matrices whose rows coincide with the columns of the matrices $A,B,G$, and $W$ corresponding to the set $J^c$. Then the condition $F \perp \text{span}(W_j, \ j \in J^c)$ can be rewritten as $F \subset \text{Ker}(\tilde{W})$.
  By Lemma \ref{l: all compressible} and \eqref{eq: kappa small},
   \begin{multline*}
    \P (F \cap S^{n-1} \not \subset \text{Incomp} (1- r^2 \d/4, K^{-C})) \\
   \le \P \left( \exists x \in \text{Comp} (1-\k/2, K^{-C}):
   \tilde{W}x =0 \right)
    \le e^{- c_3 n}.
  \end{multline*}
  Assume that $F \cap S^{n-1} \subset \text{Incomp} (1- r^2 \d/4, K^{-C})$.
  Since $\dim E =|I| \ge (r^2 \d/4) n$, the
incompressibility means that for any $y \in F \cap S^{n-1},
  \ \norm{P_E y}_2 \ge \t= K^{-C}$.
  Hence, by \eqref{eq: projection-norm} and Lemma \ref{l: projection measure},
  for $z\in \mathcal{S}^J_{q,Q}$,
  \begin{multline*}
    \P \big(   \norm{P_F (A \odot G) z}_2 < \e \text{ and } F \cap S^{n-1} \subset \text{Incomp} (1- r^2 \d/4, K^{-C}) \big)
    \\
    \le \left( \frac{c' K^{C'} \e}{\t \sqrt{n-m}} \right)^{n-m}
    \le  \left( \frac{c' K^{C''} \e}{ \sqrt{n-m}} \right)^{n-m}.
  \end{multline*}

  By Lemma \ref{l: subspace norm} and \eqref{newstar},
  \[
     \P ( \norm{P_F (A \odot G) : \R^J \to \R^n} \ge C_0 t^{-1/2}
     \sqrt{n-m}) \le e^{-  t^{-1} (n-m)}.
  \]
  Let
  \[
  \eta= \frac{\e \sqrt{t}}{2 C_0  \sqrt{n-m}}.
  \]
  By the volumetric estimate we can find an $\eta$-net $\NN^J\subset
  \mathcal{S}^J_{q,Q}$ of cardinality
  \[
   |\NN^J| \le  \left(\frac{3}{\eta} \right)^{d}.
  \]
  For $\eta$ chosen above we have
  \begin{multline*}
    \P \big( \exists z \in \NN^J :
    \   \norm{P_F (A \odot G) z}_2 < \e \text{ and } F \cap S^{n-1} \subset \text{Incomp} (1- r^2 \d/4, K^{-C}) \big)
    \\
    \le   \left(\frac{3}{\eta} \right)^{d} \cdot \left( \frac{c' K^{C''} \e }{ \sqrt{n-m}} \right)^{n-m}
    \le \left( \frac{c''K^{2C''} \e}{ \sqrt{t}\sqrt{n-m}} \right)^{(n-m)/2}.
  \end{multline*}
  This does not exceed $t^{(n-m)/4}$, if we set
  \begin{equation}   \label{eq: def epsilon}
   \e=c K^{-C} \sqrt{n-m} \cdot t.
  \end{equation}
  Assume now that
  \begin{itemize}
    \item $ \forall z \in \NN^J \   \norm{P_F (A \odot G) z}_2 \ge \e$;
    \item $F \cap S^{n-1} \subset \text{Incomp} (1- r^2 \d/4, K^{-C})$;
    \item $ \norm{P_F (A \odot G) : \R^J \to \R^n} \le C_0 t^{-1/2} \sqrt{n-m}$.
  \end{itemize}
  The previous proof shows that these conditions are satisfied with probability at least
  \[
     1-
     t^{(n-m)/4}
     - e^{- c_3 n}
     -e^{-  t^{-1} (n-m)}
     \ge      1-
     2t^{(n-m)/4}
     - e^{- c_3 n}
  \]
  Let $z' \in \mathcal{S}_{q,Q}^J$
  and let $z \in \NN^J$ be an $\eta$-approximation of $z'$:
  $\norm{z'-z}_2< \eta$. Then, on the event above,
  \begin{multline*}
    \norm{P_F (A \odot G) z'}_2
    \ge \norm{P_F (A \odot G) z}_2- \norm{P_F (A \odot G) : \R^J \to \R^n} \cdot \eta \\
    \ge \e- C_0 t^{-1/2} \sqrt{n-m} \cdot \eta
    \ge \e/2.
  \end{multline*}
  We thus have proved that
  \begin{multline*}
    \P \big( \exists z \in \mathcal{S}^J_{q,Q}
   :  \   \norm{P_F (A \odot G) z}_2 < c K^{-C}  \sqrt{n-m} \cdot t  \big) \\
    \le  2t^{(n-m)/4}
     + e^{- c_3 n}.
  \end{multline*}
  Combining this with \eqref{eq via dist},  \eqref{eq: dist F}, and
  \eqref{eq: def epsilon} we obtain
  \begin{multline*}
    \P \left(\inf_{x \in \text{Incomp}(c_0, c_1 K^{-2})} \norm{Wx}_2
    \le \a \cdot c t K^{-C} \cdot \frac{n-m}{\sqrt{n}} \right) \\
    \le \b^d \cdot \max_{J \subset [n], \ |J|=d}
    \P \big( \exists z \in \mathcal{S}_{q,Q}^J :\
    \dist(W z, H_{J^c}) < \a \cdot c t K^{-C} \cdot \sqrt{n-m} \big)
    \\
    \le \b^d \cdot \max_{J \subset [n], \ |J|=d}
    \P \big( \exists z \in \mathcal{S}_{q,Q}^J \
    \norm{P_F (A \odot G) z}_2 < \a \cdot c t K^{-C} \cdot \sqrt{n-m} \big)
    \\
    \le \b^d \cdot \left( 2t^{(n-m)/4}
     + e^{- c_3 n}   \right).
  \end{multline*}
  Recall that  $d=\lfloor (n-m) /2 \rfloor$, $\a=K^{-c}$, and $\b=K^c$. Replacing $t$ by $\b^2 t$ in the inequality above to eliminate the coefficient $\b^d$ in the right hand side, we complete the proof of the theorem.

\end{proof}

\section{Applications}
\label{sec-app}
 \subsection{Singular value bounds}
The bound on the smallest singular value in Theorem \ref{th: graph matrix} follows immediately from Theorem \ref{th: smallest singular value}  and Lemma \ref{l: subspace norm}. To bound $s_m(A \odot G)$ we apply Theorem \ref{th: intermediate singular} to the matrix $W$ consisting of the first $m$ columns of $A \odot G$ and note that $s_m(W) \le s_m(A \odot G)$.

To prove Theorem \ref{th: doubly stochastic},
decompose the matrix $W$ by writing
$W=W^{(1)}+W^{(2)}$ where $W^{(1)}$ and $W^{(2)}$
are independent centered Gaussian matrices with independent entries and
\[
   \mbox{\rm Var}(W^{(1)}_{i,j})=
   \begin{cases}
   c/n, &\text{if } c/n \le a_{i,j} \\
   0, & \text{otherwise,}
   \end{cases}\,,
   \qquad
   \mbox{\rm Var}(W^{(2)}_{i,j})= a_{i,j}- \mbox{\rm Var}(W^{(1)}_{i,j}).
\]
Let $\Om$ be the event $\norm{W^{(2)}} \le C' \log^2 n$.
By Lemma \ref{l: riemer-schutt}, for  appropriate constants $C',C''$,
one has
 \[
 \P(\Om^c)
  \le \P(\norm{W} \ge C' \log^2 n) \le \exp(-C'' \log^4 n).
 \]
On the other hand,
  \begin{multline*}
    \P (s_n(W) \le c t  n^{-1} \log^{-c'}n \text{ and } \Om)  \\
    \le E_{W^{(2)}} \P (s_n(W^{(1)}+W^{(2)}) \le c t  n^{-1} \log^{-C'}n
    \mid  \Om) \\
     \le \sup_{X: \|X\|\leq C'\log^2 n} \P (s_n(W^{(1)}+X) \le c t  n^{-1} \log^{-c'}n).
  \end{multline*}
  By Theorem \ref{th: smallest singular value} applied to
  $\sqrt{n}W^{(1)}$ and $B=\sqrt{n}X$ with $K=C' \log^2 n$, the last probability is at most $t+e^{-cn}$.
  The second estimate in Theorem \ref{th: doubly stochastic} is proved by the same argument.

\subsection{Permanent estimates}

We turn next to the proof of the theorems in the introduction.
We begin
with a refinement of Theorem \ref{th: graph permanent-intro}.
\begin{theorem} \label{th: graph permanent}
 There exist $\tilde{C}, c, c'$ depending only on $\d, \k$ such that for any $\t \ge 1$
and any adjacency matrix $A$ of a $(\d,\k)$-broadly connected graph
 \begin{eqnarray}
  \label{eq-160113}
 && \P\left( \left|
   \log{\mbox{\rm det}^2(A_{1/2}\odot G)}
   -\E\log{\mbox{\rm det}^2(A_{1/2}\odot G)}
 \right| >\tilde C(\t n \log n)^{1/3} \right)  \\
  &\leq& 6\exp(-\t)+3 \exp \left(-c \t^{1/3} n^{1/3} \log^{-2/3}n  \right) +  9 \exp \left(-c' n  \right).
\nonumber
\end{eqnarray}
and
\begin{equation}
  \label{eq-160113a}
  \E\log{\mbox{\rm det}^2(A_{1/2}\odot G)}
  \le \log \text{\rm per}(A)
   \le \E\log{\mbox{\rm det}^2(A_{1/2}\odot G)} + C' \sqrt{n \log n}.
 \end{equation}
\end{theorem}
Theorem \ref{th: graph permanent-intro}
follows from Theorem \ref{th: graph permanent} since the right
side of
\eqref{eq-160113}
does not exceed
$9\exp(-\t)+12 \exp(-c \sqrt{n} /\log n)$. The coefficients $9$ and $12$ can be removed by adjusting the constants $\tilde{C}$ and $c'$.

\begin{proof}
 The proof of Theorem \ref{th: graph permanent} is partially
based on the ideas of
 \cite[Pages 1563--1566]{FRZ}.
We would like to apply the Gaussian concentration inequality to
the logarithm of the determinant of the matrix $A_{1/2} \odot G$,
which can be written as the sum of the logarithms of its singular values.
However, since the logarithm is not a Lipschitz function, we will have
to truncate it in a neighborhood of zero in order to be able to
apply the concentration inequality. This truncation is introduced in
Section \ref{sec-trunc}.

The singular values will be divided into two groups.
For the large values of $n-l$ we use the concentration of
the (sums of subsets) singular values $s_{n-l}(A_{1/2} \odot G)$ around their
mean.
In contrast to \cite{FRZ}, we
do not use the concentration inequality once,
but rather divide the range
of singular values to several subsets, and apply separately
the concentration inequality in each subset.
The
definition of the subsets, introduced
in Section \ref{sec-trunc}, will be
chosen to match the singular values estimates of
Theorem
\ref{th: intermediate singular}.

On the other hand,
when $n-l$ becomes small, the concentration doesn't
provide an efficient estimate. In that
case we use the lower bounds for such singular values obtained
in Theorem \ref{th: smallest singular value}.
Because
the number of  singular values treated this way is small, their total
contribution to the sum of the logarithms will be small as well.
This computation is described in Section \ref{sec-7.2.2}.

Getting rid of the truncation of the logarithm
requires an a-priori rough estimate on the second moment
of
   $\log{\mbox{\rm det}^2(A_{1/2}\odot G)}$, which is presented in
   Lemma \ref{l: second moment} and proved in Section \ref{sec-7.3}.
With this, we arrive in
Section \ref{sec-7.2.3},
to the control of the deviations of
   $\log{\mbox{\rm det}^2(A_{1/2}\odot G)}$ from
   $\E\log{\mbox{\rm det}^2(A_{1/2}\odot G)}$ that is presented in
   \eqref{eq-160113}.

To complete the proof of the Theorem, we will need to relate\\
   $\E\log{\mbox{\rm det}^2(A_{1/2}\odot G)}$ to
   $\log{\E \mbox{\rm det}^2(A_{1/2}\odot G)}=\mbox{\rm perm}(A)$.
   This is achieved in Section \ref{sec-7.2.4} by again truncating the log
   (at a level different than that used before) and employing an exponential
   inequality.

\subsubsection{Construction of the truncated determinant}
\label{sec-trunc}

Let $k_* \in \N$ be a number to be specified later. We choose
truncation dimensions $n_k$ and the truncation levels $\e_k$ for
large codimensions first. For $k=0 \etc k_*$ set
 \begin{align*}
    n_k &=n \cdot 2^{-4k}; \\
    t_k &= \sqrt{\t} \cdot 2^{k+k_*}; \\
    \e_k &=
c_0 \frac{n_k}{\sqrt{n}}= c_0 \sqrt{n} \cdot 2^{-4k}.
 \end{align*}
Here, $c_0$ is a fixed constant to be chosen below.
We also set $l_*=n_{k_*}$.

 For any $n \times n$ matrix $V$ define the function $f(V)$ by
 \[
   f(V)=\sum_{k=1}^{k_*} f_k(V), \quad \text{where }
f_k(V)=\sum_{l=n-n_{k-1}}^{n-n_k-1} \log_{\e_k}(s_{n-l}(V)),
 \]
where $\log_\e(x)=\log(x\vee \e)$.
Recall that the function $S: \R^{n^2} \to \R_+^n$ defined by $S(V)=(s_1(V) \etc s_n(V))$ is $1$-Lipschitz.
 Hence, each function $f_k$ is Lipschitz
with Lipschitz constant
 \[
  L_k \le \frac{\sqrt{n_{k-1}-n_{k}}}{\e_k} \le c' \cdot 2^{2k}.
 \]
 Denote $W=A_{1/2}\odot G$.
 The concentration of the Gaussian measure implies that for an appropriately chosen constant $C$, one has
 \[
   \P \left( | f_k(W)- \E f_k(W)| > C t_k \right)
   \le  2 \exp \left(- \frac{c t_k^2}{L_k^2} \right) \le 2
\exp \left( - 2^{2(k_*-k)} \t \right).
 \]
(For this version, see e.g. \cite[Formula (2.10)]{Ledoux}.)
 Therefore,
\begin{equation}
\label{sabof4}
  \P \left(|f(W)- \E f(W)|> C \sum_{k=1}^{k_*}  t_k \right)
  \le 2 \sum_{k=1}^{k_*} \exp \left( - 2^{2(k_*-k)} \t \right) \le 4 e^{-\t}.
\end{equation}
 Here
 \[
   \sum_{k=1}^{k_*} t_k = \sum_{k=1}^{k_*} \sqrt{\t} 2^{k+k_*} \le 2 \sqrt{\t} 2^{2k_*}
   = 2 \sqrt{\t} \cdot \sqrt{\frac{n}{l_*}}.
 \]

We similarly handle singular values $s_{n-l}$ for $l\geq n-l_*$. Define
the function
$g(V)=\sum_{l=n-n_{k_*}}^n \log_{\e_{k_*}}(s_{n-l}(V))$, whose Lipschitz
constant is bounded by $\sqrt{l_*}/\e_{k^*}= c_0^{-1}\sqrt{n/l_*}$,
and therefore
\begin{equation}
\label{sabof3}
\P\left(|g(W)-\E g(W)|\geq
   c_1 \sqrt{\t} \cdot \sqrt{\frac{n}{l_*}}\right)\leq 2e^{-\tau}
\,.
\end{equation}
Set
$$\e(l)=\left\{\begin{array}{ll}
\e_k,& l\in [n_k+1,n_{k-1}]\\
\e_{k_*}, &
l\leq n_{k_*}=l_*.
\end{array}
\right.$$
Define
$$\widetilde{\mbox{\rm det}}(W, l_*)=\prod_{l=0}^{n-1}
(s_{n-l}(W)\vee \e(l))^2.$$
We include $l_*$ as the second argument to emphasize the dependence on the truncation level.
From
\eqref{sabof4} and \eqref{sabof3},
 we obtain the large deviation bound for the logarithm of the truncated determinant:
\begin{equation}
\label{satof1}
\P(|\log \widetilde{\text{det}}(W, l_*) -\E\log \widetilde{\text{det}}(W, l_*)|\geq c_2 \sqrt{\tau}\sqrt{n/l_*})
\leq 6 e^{-\tau}.
\end{equation}

\subsubsection{Basic concentration estimate for $\log \mbox{\rm det}^2(W)$}
\label{sec-7.2.2}
Our next goal is to get rid of the truncation, i.e., to relate
$\widetilde{\text{det}}(W, l_*)$ to $\mbox{\rm det}^2(W)$.
Toward this end,
 define
the set of $n \times n$ matrices  $\WW_1$ as follows:
 \[
  \WW_1= \{ V \mid \exists k, \  1\le k \le k_*,
 \ s_{n-n_k}(V) < \e_k \}.
 \]
 Then by  Theorem \ref{th: intermediate singular},
 \[
  \P(W \in \WW_1) \le \sum_{k=1}^{k_*} \left(c_0 \e_k \cdot \frac{\sqrt{n}}{n_k} \right)^{n_k} +k_* e^{-c n}
  \le 2 e^{-n_{k_*}/4},
 \]
 with an appropriate choice of the constant $c_0$.

For codimensions smaller than $l_*=n_{k_*}$ we simply estimate
the total
contribution of small singular values.
 For $0 \le l\leq  l_*$ set
 \[
   d_{n-l}=n^{-\frac{l_*}{(l+1) \log l_*} -\frac{1}{2}}.
 \]
 Let $\WW_2$ be the set of $ n \times n$ matrices defined by
 \[
  \WW_2= \{ W \mid \exists l \le l_*, \ s_{n-l}(W) \le d_{n-l} \}.
 \]
 Applying  Theorem \ref{th: smallest singular value} for $0 \le l<4$ and \ref{th: intermediate singular} for $4 \le l \le l_*$, we obtain
 \begin{align*}
   \P ( W \in \WW_2)
   &\le \sum_{l=0}^{3} c \sqrt{n} \cdot d_{n-l}
    +\sum_{l=4}^{l_*} \left(c \frac{\sqrt{n}}{l} \cdot d_{n-l} \right)^{l/4} +(l_*+1) e^{-cn} \\
   &\le C l_* \cdot n^{- \frac{l_*}{4 \log l_*}}
    \le C l_* \cdot \exp(-l_*/4) \le \exp(-l_*/8).
 \end{align*}
 Assume that $V \notin \WW_2$. Then
 \[
  \sum_{l=0}^{l_*} \log s_{n-l}^{-1}(V)
  \le \sum_{l=0}^{l_*} \log n \cdot \left(\frac{1}{2}+ \frac{l_*}{(l+1) \log l_*} \right)
  \le \frac{3}{2}l_* \log n.
 \]

Let
$\WW_3$ denote
the set of all $n \times n$ matrices $V$ such that $\norm{V} \ge n$.
Then $\P(W \in \WW_3)< e^{-n}$. If $V \notin \WW_3$, then
 \[
  \sum_{l=0}^{l_*} \log s_{n-l}(V)
  \le l_* \log n.
 \]
 Therefore, for any $V\in (\WW_2\cup \WW_3)^c$,
   $$-\frac{3}{2} l_*\log n\leq \sum_{l=0}^{l_*}\log s_{n-l}(V)
   \leq
   \sum_{l=0}^{l_*}\log (s_{n-l}(V)\vee \e_{k_*})\leq l_*
   \log n\,.
$$
We thus obtain that
if $W\in (\WW_1\cup \WW_2\cup \WW_3)^c$ then
\begin{equation}
\label{satof1b}
|\log \mbox{\rm det}^2(W)-\log \widetilde{\text{det}}(W, l_*)|\leq \frac32 l_*\log n
\end{equation}
Note that the event $W\in (\WW_1\cup \WW_2\cup\WW_3)^c$ has probability
larger than $1-3e^{-l_*/8}$.

Setting
$$Q(l_*)=\E\log
\widetilde{\text{det}}(W, l_*),$$
we thus conclude from \eqref{satof1}
that
\begin{multline}
  \label{eq-120912h}
  \P\left(\left|
  \log \mbox{\rm det}^2(W)-Q(l_*)\right|
  \geq \frac32 l_*\log n+c_2\sqrt{\t n/l_*}
  \right) \\
  \leq 6\exp (-\t)+3 \exp(-l_*/8) \,.
\end{multline}
 This is our main concentration estimate.
 We will use it with $l_*$ depending on $\t$ to obtain an optimized concentration bound. Also, we will use special choices of $l_*$  to relate a hard to evaluate  quantity $Q(l_*)$ to the characteristics of the distribution of $ \mbox{\rm det}^2 (W)$, namely to $\E \log \mbox{\rm det}^2(W)$ and $\log \E \mbox{\rm det}^2(W)$.
 This will be done by comparing  $\E \log \mbox{\rm det}^2(W)$ to $Q(l_1)$ and  $\log \E \mbox{\rm det}^2(W)$ to $Q(l_2)$  for different values $l_1$ and $l_2$. This means that we also have to compare $Q(l_1)$ and $Q(l_2)$. The last comparison requires only \eqref{eq-120912h}.

 Let $100 \le l_1, l_2 \le n/2$. For $j=1,2$, denote
 \[
   \tilde{W}_j
   = \left\{ V \Big| |\log \mbox{\rm det}^2 (V)-Q(l_j)|
     \le \frac{3}{2} l_j \log n + 4 c_2 \sqrt{n/l_j} \right\}.
 \]
  Using \eqref{eq-120912h} with $\t=16$, we show that $\P(\tilde{W}_j) >1/2$ for $j=1,2$. This means that $\tilde{W}_1 \cap \tilde{W}_2 \neq \emptyset$. Taking $V \in \tilde{W}_1 \cap \tilde{W}_2$, we obtain
  \begin{align} \label{eq: l_1 and l_2}
   |Q(l_1)- Q(l_2)|
   &\le |Q(l_1)- \log \mbox{\rm det}^2 (V)| + |\log \mbox{\rm det}^2 (V)- Q(l_2)| \\
   &\le \frac{3}{2}(l_1+l_2) \log n + c n^{1/2}(l_1^{-1/2}+l_2^{-1/2}).
   \notag
  \end{align}

\subsubsection{Comparing $Q(l_*)$ to $\E \log \mbox{\rm det}^2(W)$}
\label{sec-7.2.3}
 Our next task is to relate  $\E \log \mbox{\rm det}^2(W)$ to $Q(l_*)$ for some $l_*=l_1$. Toward this end we
 optimize the  left side of  \eqref{eq-120912h}  for $\t=8$ by choosing $l_*=l_1$, where
 \[
  2 n^{1/3} \log^{-2/3}n \le l_1 =n \cdot 2^{-4k_1} < 32  n^{1/3} \log^{-2/3}n.
 \]
Then we get from \eqref{eq-120912h} that there exists $c>0$ such that for all $\t \ge 1$,
\begin{multline}
  \label{eq-120912j}
  \P\left(\left|
  \log \mbox{\rm det}^2(W)-Q(l_1)\right|\geq c \tau^{1/2} ( n\log n)^{1/3}\right)\\
  \leq 6\exp(-\t)+3\exp(-l_1 /8)\,.
\end{multline}

  Let $\WW_4$ be the set of all $n \times n$ matrices $V$
  such that $|\log \det (V)^2 -Q(l_1)| > \sqrt{n}$.
  The inequality \eqref{eq-120912j}
  applied with $\t=c' l_1$ for an appropriate $c'$  reads
 \begin{equation}\label{WW_4}
  \P(W \in \WW_4)
  \le  \exp \left(- c l_1  \right)
  = \exp \left(- C n^{1/3} \log^{-2/3}n  \right).
 \end{equation}
 We have
 \begin{align*}
  &|\E \log \mbox{\rm det}^2(W)-Q(l_1)|
  \le  \E | \log \mbox{\rm det}^2(W)-Q(l_1)| \\
  =&\E | \log \mbox{\rm det}^2(W)-Q(l_1)|\cdot \mathbf{1}_{\WW_4^c}(W) + \E | \log \mbox{\rm det}^2(W)-Q(l_1)|\cdot \mathbf{1}_{\WW_4}(W).
 \end{align*}
 The first term here can be estimated by integrating the tail in
 \eqref{eq-120912j}:
 \begin{align*}
  &\E | \log \mbox{\rm det}^2(W)-Q(l_1)|\cdot \mathbf{1}_{\WW_4^c}(W) \\
  &\le c ( n \log n)^{1/3} + \int_{c ( n \log n)^{1/3}}^{\sqrt{n}}  \P( | \log \mbox{\rm det}^2(W)-Q(l_1)|>x) \, dx \\
  &\le c ( n \log n)^{1/3} + \int_1^{c' l_1} 2 \exp \left(- \left(\frac{ x}{c(n \log n)^{1/3}} \right)^2\right) \, dx
  \le C ( n \log n)^{1/3}.
 \end{align*}
 To bound the second term, we need the following
rough estimate of the second moment of the logarithm of the determinant.
 The proof of this estimate will be presented in the next subsection.
 \begin{lemma} \label{l: second moment}
  Let $W=G \odot A'_{1/2}$, where $G$ is the standard Gaussian matrix, and $A'$ is a deterministic matrix with entries $0 \le a_{i,j} \le 1$ for all $i,j$ having at least one generalized diagonal with entries $a'_{i, \pi(i)} \ge c/n$ for all $i$.
  Then
  \[
    \E \log^2 \mbox{\rm det}^2(W) \le \bar{C} n^3.
  \]
 \end{lemma}
 Since $A$ is the matrix of a $(\d, \k)$-broadly connected graph, it satisfies the conditions of Lemma \ref{l: second moment}.
 The estimate of the second term follows from Lemma \ref{l: second moment}, \eqref{eq-120912j}, and the Cauchy--Schwarz inequality:
 \begin{align*}
    &\E | \log \mbox{\rm det}^2(W)-Q(l_1)|\cdot \mathbf{1}_{\WW_4}(W) \\
   &\le  \left( \E  |\log \mbox{\rm det}^2(W)-Q(l_1)|^2 \right)^{1/2} \cdot \P^{1/2}(W \in \WW_4) \\
   & \le (\bar{C}n^3+2 Q^2(l_1))^{1/2} \cdot  \exp \left(- (C/2) n^{1/3} \log^{-2/3}n  \right).
 \end{align*}
 Combining the bounds for $\WW_4$ and $\WW_4^c$, we get
 \begin{multline*}
  |\E \log \mbox{\rm det}^2(W)-Q(l_1)|
  \\ \le  C ( n \log n)^{1/3}+ (\bar{C} n^3+2 Q^2(l_1))^{1/2} \cdot  \exp \left(- (C/2) n^{1/3} \log^{-2/3}n  \right),
 \end{multline*}
 which implies
 \begin{equation} \label{eq: expectation log}
  |\E \log \mbox{\rm det}^2(W)-Q(l_1)| \le  C' ( n \log n)^{1/3}.
 \end{equation}

 \subsubsection{Comparing   $\log \E \, \mbox{\rm det}^2(W)$ to $\E \log  \mbox{\rm det}^2(W)$ }
 \label{sec-7.2.4}
 We start with relating $Q(l_1)$ and $\log \E \mbox{\rm det}^2(W)=\log \text{perm}(A)$.
 To this end we will use a different
 value of $l_*$. Namely, choose $l_2$ so that
 \[
  \sqrt{n/\log n} \le l_2=n \cdot 2^{4 k_2} < 16 \sqrt{n/\log n}.
 \]
 The reasons for this choice will become clear soon.
 Denote for
 brevity
 \[
   U:= \log \widetilde{\text{det}}(W, l_2) -\E\log \widetilde{\text{det}}(W, l_2).
 \]
 We deduce from
 \eqref{satof1} that
\begin{eqnarray*}
\E(e^{U})&\leq &
\E(e^{|U|})\leq 1+\int_0^\infty e^t \P(|U|\geq t) \, dt\\
&\leq & 1+ 6\int_0^\infty e^t e^{-t^2l_2/c_2n} \,dt
\leq 1+c_3e^{c_4 n/l_2}\,.
\end{eqnarray*}
Taking logarithms,
we conclude that
\begin{eqnarray*}
 \log \E  \mbox{\rm det}^2(W) &\leq &
 \log \E \widetilde{\text{det}}(W, l_2) \\
&\leq &
\E \log \widetilde{\text{det}}(W, l_2)+
\log (1+c_3e^{c_3 n/l_2})\nonumber\\
&\leq&
Q(l_2)+c_4 n/l_2\,.
\nonumber
\end{eqnarray*}
The inequality \eqref{eq: l_1 and l_2} implies
\begin{align} \label{eq: Q-perm 1a}
  &\log \E  \mbox{\rm det}^2(W)  \\
  &\le Q(l_1)+ c_4 n/l_2 +\frac{3}{2}(l_1+l_2) \log n
      + c n^{1/2}(l_1^{-1/2}+l_2^{-1/2})  \notag \\
  &\le Q(l_1)+ c_5 \sqrt{n \log n}.
  \notag
\end{align}
The value of $l_2$ was selected to optimize the inequality \eqref{eq: Q-perm 1a}.
To bound $Q(l_1)-\log \E  \mbox{\rm det}^2(W)$ from above, we use \eqref{eq-120912j} with $\t=4$  to derive
\begin{multline}  \label{eq: Q-perm 2}
  \P\left( |\log \mbox{\rm det}^2(W)-Q(l_1)| \le 2 c( n \log n)^{1/3} \right) \\
  \geq 1- 6e^{-4} -3e^{-l_1/8}> \frac{1}{2}.
\end{multline}
On the other hand,
 Chebyshev's inequality applied to the random variable
$\mbox{\rm det}^2(W)/\E\mbox{\rm det}^2(W)$ implies that
$$
  \P(\mbox{\rm det}^2(W)\leq 2 \E\mbox{\rm det}^2(W))
  \geq \frac{1}{2},
$$
and therefore
\begin{equation}
\label{eq: Q-perm 3}
\P(\log \mbox{\rm det}^2(W)-\log \E\mbox{\rm det}^2(W)\leq \log 2)
  \geq \frac{1}{2}.
\end{equation}
This means that the events in \eqref{eq: Q-perm 2} and \eqref{eq: Q-perm 3} intersect, and so
\[
  Q(l_1)-\log \E  \mbox{\rm det}^2(W)
  \le 2c ( n \log n)^{1/3}+ \log 2.
\]
Together with \eqref{eq: Q-perm 1a} this provides a two-sided bound
\begin{align*}
  |Q(l_1)-\log \E  \mbox{\rm det}^2(W) |
  &\le \max \left( c_5 \sqrt{n \log n}, \  2c ( n \log n)^{1/3}+ \log 2 \right) \\
  &= c_5 \sqrt{n \log n}
\end{align*}
for a sufficiently large $n$.
The combination of this inequality with \eqref{eq: expectation log} yields
\[
  |\E  \log  \mbox{\rm det}^2(W) -\log \E  \mbox{\rm det}^2(W) |
  \le  c_6 \sqrt{n \log n}.
\]

\subsubsection{Concentration around \  $\E  \log  \mbox{\rm det}^2(W)$}
 To finish the proof we have to derive the concentration inequality. This will be done by choosing the truncation parameter $l_*$ depending on $\t$. Namely, assume first that $1 \le \t \le n^2 \log^2 n$ and define $l_*$ by
  \[
   2^{-8} \t^{1/3} n^{1/3} \log^{-2/3}n < l_* = n \cdot 2^{-4k_*} \le 2^{-4} \t^{1/3} n^{1/3} \log^{-2/3}n.
 \]
 The constraint on $\t$ is needed to guarantee that $k_*\ge 1$. Substituting  this $l_*$ in \eqref{eq-120912h}, we get
\begin{multline*}
  \P\left(\left|  \log \mbox{\rm det}^2(W)-Q(l_*)\right|
  \geq \frac32 l_*\log n+c_2\sqrt{\t n/l_*}
  \right) \\
  \leq 6\exp(-\t)+3 \exp \left(-c \t^{1/3} n^{1/3} \log^{-2/3}n  \right).
\end{multline*}
By \eqref{eq: expectation log} and \eqref{eq: l_1 and l_2}, for such $\t$ we have
\begin{align*}
    &|\E \log \mbox{\rm det}^2(W)-Q(l_*)| \\
   &\le |\E \log \mbox{\rm det}^2(W)-Q(l_1)| + |Q(l_1)-Q(l_*)| \\
   &\le    C' ( n \log n)^{1/3}  +  \frac{3}{2}(l_1+l_*) \log n + c n^{1/2}(l_1^{-1/2}+l_*^{-1/2}) \\
   &\le C'' (\t n \log n)^{1/3}.
\end{align*}
Together with the previous inequality, this implies
\begin{multline*}
  \P\left(\left|  \log \mbox{\rm det}^2(W)-\E \log \mbox{\rm det}^2(W)\right|
  \geq \tilde{C} (\t n \log n)^{1/3}
  \right) \\
  \leq 6\exp(-\t)+3 \exp \left(-c \t^{1/3} n^{1/3} \log^{-2/3}n  \right),
\end{multline*}
if the constant $\tilde{C}$ is chosen large enough.

If $\t>\t_0:=n^2 \log^2 n$, we use the  inequality above with $\t=\t_0$ and obtain
\begin{multline*}
  \P\left(\left|  \log \mbox{\rm det}^2(W)-\E \log \mbox{\rm det}^2(W)\right|
  \geq \tilde{C} (\t n \log n)^{1/3}
  \right)
  \leq 9 \exp \left(-c' n  \right),
\end{multline*}
Finally, for all $\t \ge 1$, this implies
\begin{multline*}
  \P\left(\left|  \log \mbox{\rm det}^2(W)-\E \log \mbox{\rm det}^2(W)\right|
  \geq \tilde{C} (\t n \log n)^{1/3}
  \right) \\
  \leq 6\exp(-\t)+3 \exp \left(-c \t^{1/3} n^{1/3} \log^{-2/3}n  \right) +  9 \exp \left(-c' n  \right).
\end{multline*}
which completes the proof of Theorem \ref{th: graph permanent}.
\end{proof}

\subsection{Second moment of the logarithm of the determinant}
\label{sec-7.3}
It remains to prove Lemma \ref{l: second moment}. The estimate of the
lemma, which was necessary in the proof of \eqref{eq-120912},
is very far from being precise, so we will use rough, but elementary bounds.
 \begin{proof}[Proof of Lemma \ref{l: second moment}]
  We will estimate the expectations of the squares of the positive and negative parts of the logarithm separately.
  Denote by $W_1 \etc W_n$ the columns of the matrix $W$. By the Hadamard inequality,
  \begin{multline*}
    \E \log_+^2 \det (W)^2 \le \sum_{j=1}^n \E \log_+^2 \norm{W_j}_2^2
    \le n \sum_{j=1}^n \sum_{i=1}^n \E \log^2 (1+ w_{i,j}^2) \le Cn^3.
  \end{multline*}
  Here in the second inequality we used an elementary bound
  \[
    \log_+ ( \sum_{i=1}^n u_i) \le \sum_{i=1}^n \log (1+u_i)
  \]
  valid for all $u_1 \etc u_n \ge 0$, and the Cauchy--Schwarz inequality. The last inequality holds since $w_{i,j}$ is a normal random variable of variance at most $1$.

  To prove the bound for $\E \log_-^2 \det (W)^2$, assume that $a'_{i,i} \ge c/n$ for all $i \in [n]$.
   Set $A''=\sqrt{n/c}A'_{1/2}$, so $a_{i,i}'' \ge 1$, and let $W''=A'' \odot G$. Then $\E \log_-^2 \det (W)^2 \le \E \log_-^2 \det (W'')^2 + 2 n \log n$. We will prove the following estimate by induction:
   \begin{equation}\label{induction}
     \E \log_-^2 \det (W'')^2 \le c' n^2,
   \end{equation}
   where the constant $c'$ is chosen from the analysis of the one-dimensional case.

  For $n=1$ this follows from the inequality
  \begin{equation}\label{one-dim}
    \E \log_-^2 (w_{1,1}+x) \le c',
  \end{equation}
  which holds for all $x \in \R$. Assume that \eqref{induction} holds for n. Denote by $\E_1$ the expectation with respect to $g_{1,1}$ and by $\E'$ the expectation with respect to $G^{(1)}$, which will denote the other entries of $G$. Denote by $D_{1,1}$ the minor of $W''$ corresponding to the entry $(1,1)$.
  Note that $D_{1,1} \neq 0$ a.s.
  Decomposing the determinant with respect to the first row, we obtain
  \begin{multline*}
     \E \log_-^2 \det (W'')^2
      =\E' \left( \E_1 \left[ \log_-^2 (a_{1,1}'' g_{1,1} D_{1,1} +Y) \mid G^{(1)} \right] \right) \\
     =\E' \left( \E_1 \left[  \left(\log_-(a_{1,1}'' g_{1,1} +\frac{Y}{D_{1,1}}) + \log_-(D_{1,1})  \right)^2 \mid G^{(1)} \right] \right).
  \end{multline*}
  Since $Y/D_{1,1}$ is independent of $g_{1,1}$, inequality \eqref{one-dim} yields
  \[
     \E_1 \left(  \log_-^2 \left((a_{1,1}'' g_{1,1} +\frac{Y}{D_{1,1}} \right) \mid G^{(1)}  \right) \le c.
  \]
  Therefore, by Cauchy--Schwarz inequality,
  \begin{multline*}
     \E \log_-^2 \det (W'')^2 \\
     \le \E'   \left( c'+2 \E_1 \left[  \log_- \left(a_{1,1}'' g_{1,1} +\frac{Y}{D_{1,1}} \right) \mid G^{(1)} \right] \cdot \log_-(D_{1,1}) +\log_-^2(D_{1,1})   \right) \\
     \le \left(\sqrt{c'}+ \sqrt{ \E' \log_-^2(D_{1,1})} \right)^2.
  \end{multline*}
  By the induction hypothesis, $\E' \log_-^2(D_{1,1}) \le c' n^2$, so
  \[
     \E \log_-^2 \det (W'')^2 \le    c'(n+1)^2.
  \]
  This proves the induction step, and thus completes the proof of Lemma \ref{l: second moment}.
 \end{proof}

  Theorem \ref{th: matrix permanent-intro} is proved similarly, using this time
  Theorem
  \ref{th: doubly stochastic} instead of Theorem \ref{th: graph matrix}, and taking
into account the degradation of the Lipschitz constant due to the presence of $b_n$. We omit further details.

\subsection{Concentration far away from permanent}
 Consider an approximately doubly stochastic matrix $B$ with all entries of order $\Omega(n^{-1})$, which has some entries of order $\Omega(1)$. For such matrices the conditions of Theorem \ref{th: matrix permanent-intro} are satisfied with $\d,\k=1$, so the Barvinok--Godsil-Gutman estimator is strongly concentrated around $\E \log \text{det}^2 (B_{1/2} \odot G)$. Yet, the second inequality of this Theorem reads
 \[
    \log \text{per}(B)
   \le \E \log \text{det}^2 (B_{1/2} \odot G)
   +C'n \log^{c'}c n,
 \]
 which is too weak  to obtain a subexponential deviation of the estimator from the permanent. However, the next lemma shows that the inequality above is sharp up to a logarithmic term. This means, in particular, that
 the
 Barvinok--Godsil-Gutman estimator for such matrices can be concentrated around a value, which is $\exp(cn)$ away from the permanent.

\begin{lemma}
  \label{l: example}
  Let $\a>0$, and let $B$ be an $n \times n$ matrix with entries
  \[
   b_{i,j}=
   \begin{cases}
    \a/n, &\text{for } i \neq j; \\
    1,    &\text{for } i=j.
   \end{cases}
  \]
  There exist constants $\alpha_0,\beta>0$ so that
  if $0< \alpha<\alpha_0$ then
  \begin{equation}
    \label{eq-12of1}
    \liminf_{n\to\infty}
    \frac1n \left| \E\log \det(B_{1/2}\odot G)^2-
    \log \E \det (B_{1/2}\odot G)^2\right|\geq \beta
    .
  \end{equation}
\end{lemma}
\begin{proof}
  Recall that from \eqref{eq-120912b}, we have that for any fixed $\alpha<1$,
  the random variable
    $$
    \frac1n \left| \E\log \det(B_{1/2}\odot G)^2-
    \log  \det (B_{1/2}\odot G)^2\right|$$
    converges to $0$ (in probability and a.s.).
    Since
    $$\E \det (B_{1/2}\odot G)^2=\mbox{\rm per}(B)\geq 1\,,$$
    it thus suffices to show that, with constants as in the statement of the
    lemma,
    \begin{equation}
      \label{eq-12of2}
    \liminf_{n\to\infty}
    \frac1n  \log \det(B_{1/2}\odot G)^2\leq -\beta\,, \quad a.s.\,.
  \end{equation}
  We rewrite the determinant as a sum over permutations with $\ell$
  fixed points.
  We then have
$$ \det(B_{1/2}\odot G)=
\sum_{\ell=0}^n \sum_{F\subset [n], |F|=\ell}
\left(\prod_{i\in F}G_{ii}\right)
\frac{(-1)^{\sigma(F)}
M_F \alpha^{(n-\ell)/2}}{n^{(n-\ell)/2}}
=: \sum_{\ell=0}^n A_\ell\,,$$
where $M_F$ is the determinant of an $(n-\ell)\times (n-\ell)$
matrix with i.i.d. standard Gaussian
entries, $EM_F^2=(n-\ell)!$, $\sigma(F)$ takes values in $\{-1,1\}$
and $M_F$ is independent of
$\prod_{i\in F}G_{ii}$. (Note that $M_{F_1}$ is not independent of
$M_{F_2}$ for $F_1\neq F_2$.)

Recall that
\begin{equation}
  \label{eq-12of3}
  \binom{n}{\ell} \leq  e^{nh(\ell_n)}\,,
\end{equation}
where $\ell_n=\ell/n$ and $h$ is the entropy function,
$h(x)=-x\log x-(1-x)\log(1-x)\leq \log 2$.

We will need the following easy consequence of Chebyshev's inequality: for
any $y>0$,
\begin{equation}
  \label{eq-12of4}
  \P(|\prod_{i=1}^\ell G_{ii}|\geq e^{-y\ell})
  \leq (\E|G_{11}|)^\ell e^{y\ell}=
  \left(\sqrt{\frac{2}{\pi}}e^y\right)^\ell\,.
\end{equation}
It is then clear that there exist $\delta_1,\delta_2>0$ so that,
for any $\ell_n>(1-\delta_1)$, one has
\begin{equation}
  \label{eq-12of3a}
  \binom{n}{\ell}
  \P(|\prod_{i=1}^\ell G_{ii}|\geq e^{-\delta_2 n})
  \leq \frac{1}{n^3}\,.
\end{equation}
Choose now $\delta_1'\leq \delta_1$ positive so that
\begin{equation}
  \label{eq-12of4a}
  \delta_2>3h(1-\delta_1')\,,
\end{equation}
which is always possible since $h(\cdot)$ is continuous and $h(1)=0$.

We will show that we can find $\alpha_0>0$ such that for any
$\alpha<\alpha_0$, for all $n$ large and any $\ell$,
\begin{equation}
  \label{eq-12of5}
  \P(|A_\ell|\geq e^{-\delta_2n/2}) \leq \frac{2}{n^3}\,.
\end{equation}
This would imply \eqref{eq-12of2} and conclude the proof of the lemma.

To see \eqref{eq-12of5}, we argue separately for $\ell_n\geq (1-\delta_1')$
and $\ell_n<(1-\delta_1')$. In either case, we start with the inequality
\begin{multline}
  \label{eq-12of6}
  \P(|A_\ell| \geq e^{-\delta_2n/2})\\
  \leq \binom{n}{\ell}
  \P\left(\left(\frac{\alpha}{n}\right)^{n(1-\ell_n)/2}
  \left(\prod_{i=1}^\ell|G_{ii}|\right)
  |M_{[\ell]}|\geq \binom{n}{\ell}^{-1} e^{-\delta_2 n/2}\right).
\end{multline}
Considering first $\ell_n\geq (1-\delta_1')$, we estimate the right side in
\eqref{eq-12of6} by
\begin{align}
  \label{eq-12of7}
  &\binom{n}{\ell} \P(\left(\prod_{i=1}^\ell|G_{ii}|\geq e^{-\delta_2 n}\right)
  \\+
   &\binom{n}{\ell}
  \P\left(\left(\frac{\alpha}{n}\right)^{n(1-\ell_n)/2}
  |M_{[\ell]}|\geq \binom{n}{\ell}^{-1} e^{\delta_2 n/2}\right)
 \,.  \notag
 \end{align}
 The first term in \eqref{eq-12of7} is bounded by $1/n^3$ by our choice of
 parameters, see
 \eqref{eq-12of3a}. To analyze the second term we use Chebyshev's inequality
 and the fact that $\alpha<1$:
 \begin{align*}
  & \binom{n}{\ell}
  \P\left(\left(\frac{\alpha}{n}\right)^{n(1-\ell_n)/2}
  |M_{[\ell]}|\geq \binom{n}{\ell}^{-1}e^{\delta_2 n/2}\right) \\
  \leq
  &\binom{n}{\ell}^3 e^{-\delta_2n}\frac{\alpha^{n(1-\ell_n)} (n-\ell)!}{n^{n-\ell}}\\
  \leq
  &e^{n [3h(\ell_n)-\delta_2]}\leq e^{3h(1-\delta_1')-\delta_2}
  \leq \frac{1}{n^3}\,,
\end{align*}
where the last inequality is due to \eqref{eq-12of4a}.
This completes the proof of \eqref{eq-12of5} for $\ell_n\geq (1-\delta_1')$,
for any $\alpha\leq 1$.

It remains to analyze the case $\ell_n<(1-\delta_1')n$. This is where the
choice of $\alpha_0$ will be made. Starting from
\eqref{eq-12of6} we have by Chebyshev's inequality
\begin{eqnarray*}
  \P(|A_\ell| \geq e^{-\delta_2n/2})
&\leq &\binom{n}{\ell}^3 e^{\delta_2 n}
  \left(\frac{\alpha}{n}\right)^{n(1-\ell_n)}
  E|M_{[\ell]}|^2
  \\&
  \leq& \alpha^{n(1-\ell_n)}e^{3n\log 2}\leq
  e^{n[3\log 2+\delta_1'\log \alpha]}\,.
\end{eqnarray*}
  Choosing $\alpha_0<1$ such that
  $3\log 2+\delta_1'\log\alpha_0<0$ shows that the last term is bounded
  by $1/n^3$ for large $n$, and completes the proof of the lemma.
\end{proof}

\end{document}